\documentclass[11pt]{article}

\usepackage{times,amssymb,amsmath,exscale,array,latexsym}
\usepackage{graphicx}
\usepackage{epsfig}
\usepackage{mathrsfs}
\usepackage{mathtools}
\usepackage{stmaryrd}

\usepackage{color}
\definecolor{marin}{rgb}   {0.,   0.3,   0.7}
\definecolor{rouge}{rgb}   {0.8,   0.,   0.}
\definecolor{sepia}{rgb}   {0.8,   0.5,   0.}
\usepackage[colorlinks,citecolor=marin,linkcolor=rouge,
            bookmarksopen,
            bookmarksnumbered
           ]{hyperref}

\newcommand{\C}{\mathbb{C}}

\newcommand{\e}{\ensuremath{\mathrm{e}}}

\newcommand{\ts}{h}

\newtheorem{lemma}{Lemma}[section]

\newtheorem{theorem}[lemma]{Theorem}
\newtheorem{proposition}[lemma]{Proposition}
\newtheorem{corollary}[lemma]{Corollary}

\numberwithin{equation}{section}

\newcommand{\QED}{\mbox{}\hfill \raisebox{-0.2pt}{\rule{5.6pt}{6pt}\rule{0pt}{0pt}}
          \medskip\par}

\newenvironment{proof}{\noindent{\bf Proof:}}{$\Box$}

\title{Symmetric-conjugate splitting methods for linear unitary problems} 

\author{J. Bernier$^{1}$, S. Blanes$^{2}$, F. Casas$^{3}$\footnote{Corresponding author}, A. Escorihuela-Tom\`as$^{4}$ \\[2ex]
$^{1}$ {\small\it Nantes Universit\'e, CNRS, Laboratoire de Math\'ematiques Jean Leray, LMJL, F-44000 Nantes, France}\\{
\small\it email: joackim.bernier@univ-nantes.fr}\\[1ex]
$^{2}$ {\small\it Universitat Polit\`ecnica de Val\`encia, Instituto de Matem\'atica Multidisciplinar, 46022-Valencia, Spain}\\{
\small\it email: serblaza@imm.upv.es}\\[1ex]
$^{3}$ {\small\it Departament de Matem\`atiques and IMAC, Universitat Jaume I, 12071-Castell\'on, Spain}\\{
\small\it email: Fernando.Casas@mat.uji.es}\\[1ex]
$^{4}$ {\small\it Departament de Matem\`atiques, Universitat Jaume I, 12071-Castell\'on, Spain}\\{
\small\it email: alescori@uji.es}\\[1ex]
}

\begin{document}
\mathsurround 0.8mm

\maketitle

\begin{abstract}

We analyze the preservation properties of a family of reversible splitting methods when they are applied to
the numerical time integration of linear differential equations defined in the unitary group. The schemes involve
complex coefficients and are conjugated to unitary transformations for sufficiently small values of the time step-size.
New and efficient methods up to order six are constructed and tested on the linear Schr\"odinger equation.

\

\

\\[1ex]
\end{abstract}

%%%%%%%%%%%%%%%%%%%%%%

\section{Introduction}

%\paragraph{Class of problems.}

We are concerned in this work with the numerical integration of the linear ordinary differential equation 
\begin{equation}  \label{problem1}
  i \frac{d u}{dt} + H u = 0, \qquad\quad u(0) = u_0,
\end{equation}
where $u \in \mathbb{C}^N$ and $H \in \mathbb{R}^{N \times N}$ is a real matrix.  A particular example of paramount importance leading to eq.
(\ref{problem1}) is the time-dependent
Schr\"odinger equation once it is discretized in space. In that case $H$ (related to the Hamiltonian of the system) can be typically split into two parts,
$H = A + B$. The equation
\[
  y^{\prime\prime} + K y = 0
\]
with $y \in \mathbb{R}^d$,   $K \in \mathbb{R}^{d \times d}$ can also be recast in the form (\ref{problem1}) if the matrix $K$ satisfy certain conditions \cite{blanes08otl}.

%\paragraph{Class of methods.}

Although the solution of (\ref{problem1}) is given by $u(t) = \e^{i t H} u_0$, very often the dimension of $H$ is so large that evaluating directly the action of
the matrix exponential on $u_0$
is  computationally very expensive, and so other approximation techniques are desirable. When $H = A + B$ and  $\e^{i t A} u_0$, $\e^{i t B} u_0$ 
can be efficiently evaluated,
then splitting methods constitute a natural option \cite{mclachlan02sm}. They are of the form
\begin{equation} \label{split1}
  S_h = \e^{i h a_0 A} \, \e^{i h b_0 B} \, \cdots  \, \e^{i h b_{2n-1} B} \, \e^{i h a_{2n} A} 
\end{equation}
for a time step $h$. Here $a_j$, $b_j$ are coefficients chosen in such a way that $S_h = \e^{i h H} + \mathcal{O}(h^{p+1})$ when $h \rightarrow 0$
for a given $p \ge 1$. After applying the Baker--Campbell--Hausdorff (BCH) formula, $S_h$ can be formally expressed as
$S_h=\exp\left(i h H_h \right)$, with $iH_h=i H_h^o+H_h^e$ and
\begin{eqnarray*}
 H^o_h & = & (g_{1,1}A+g_{1,2}B)+h^2(g_{3,1}[A,[A,B]]+g_{3,2}[B,[A,B]])+\ldots \\
 H^e_h & = & h g_{2,1}[A,B]+h^3 (g_{4,1}[A,[A,[A,B]]]+\ldots)+\ldots
\end{eqnarray*} 
Here $[A,B] := AB-BA$, $g_{k,j}$ are polynomials of degree $k$ in the coefficients $a_i,b_i$ verifying $g_{1,1}=g_{1,2}=1$ (for consistency), and $g_{k,j}=0, \ k=1,2,\ldots,p, \ \forall j$ 
for achieving order $p$.

If $A$ and $B$ are real symmetric matrices, then $[A,B]$ is skew-symmetric and $[A,[A,B]]$ is symmetric. In general, all nested commutators with an even number of matrices $A,B$ are skew-symmetric and those containing an odd number are symmetric, so that $(H_h^o)^T=H_h^o$ and  $(H_h^e)^T=-H_h^e$. 

When the coefficients $a_j, b_j$ are real, then $g_{k,j}$ are also real and therefore $S_h=\e^{ih H_h}$ is a unitary matrix. In addition, if the composition (\ref{split1}) is
palindromic, i.e., $a_{2n-j} = a_j$, $b_{2n-1-j} = b_j$, $j=1,2,\ldots$, then $g_{2k,j}=0$ and $H_{-h}=H_h$, thus leading to a time-reversible method, 
 $S_{-h}=S_h^{-1}$. In other words, if $u_n$ denotes the approximation at time $t = n h$, then $S_{-h}(u_{n+1}) = u_n$. 
 As a result, one gets a very favorable long-time behavior of the error for this type of integrators \cite{lubich08fqt}.
Thus, in particular, 
\[
    \mathcal{M}(u) := |u|^2  \qquad\qquad \mbox{ (norm) }
\]    
and 
\[
   \mathcal{H}(u) := \bar{u}^T H u \qquad\qquad \mbox{(expected value of the energy)}
\]   
are almost globally preserved. %preserved for very long times.

Recently, some preliminary results obtained with a different class of splitting methods (\ref{split1}) have been reported when they are applied to the semi-discretized 
Schr\"odinger equation  \cite{blanes22asm}. These schemes are characterized by the fact that the coefficients in (\ref{split1}) are \emph{complex numbers}.
Notice, however, that in this case the polynomials $g_{k,j}\in \C$, so that $S_h=\e^{ih H_h}$ is \emph{not} unitary in general. This is so even for palindromic
compositions, since $g_{2\ell + 1,j}$ are complex anyway.

There is nevertheless a special symmetry in the coefficients, namely
\begin{equation} \label{sc1}
  a_{2n-j} = \overline{a}_j \qquad \mbox{ and } \qquad b_{2n-1-j} = \overline{b}_j, \qquad j=1,2,\ldots,
\end{equation}
worth to be considered. Methods of this class can be properly called \emph{symmetric-conjugate} compositions. In that case, a 
straightforward computation shows that the resulting composition satisfies 
\begin{equation} \label{sc2}
  \overline{S}_h = S_h^{-1}
\end{equation}
for real matrices $A$ and $B$, and in addition
\begin{equation} \label{sc3}
 (\overline{S}_h)^T  = S_{-h}
\end{equation}
if $A$ and $B$ are real symmetric. In consequence, 
\[
  iH_h = i(H+\hat H_h^o) +  i \hat H_h^e
\] 
for certain real matrices $\hat H_h^o$ (symmetric), and $\hat H_h^e$ (skew-symmetric). Since $i \hat H_h^e$ is not real, then unitarity is lost. 
 In spite of that, the
examples collected in \cite{blanes22asm} seem to indicate that this class of schemes behave 
as compositions with real coefficients regarding preservation properties, at least for sufficiently small values of $h$. 
Intuitively, this can be traced back to the fact that $i \hat H_h^e= {\cal O}(h^{p})$ and is purely imaginary.

One of the purposes of 
this paper is to provide a rigorous justification of this behavior by generalizing the treatment done in \cite{blanes22asm} for the problem (\ref{problem1})
defined in the group SU(2), i.e., when $H$ is a linear combination of Pauli matrices. In particular, we prove here
that, typically, \emph{any consistent symmetric-conjugate splitting method applied to (\ref{problem1}) when $H$ is real symmetric, is
conjugated to a unitary method for sufficiently small values of $h$}. In fact, this property can be related to the reversibility of the map $S_h$ with respect to complex conjugation, as specified next.

Let $C$ be the linear transformation defined by $C(u) = \overline{u}$ for all $u \in \mathbb{C}^N$. Then, the differential equation (\ref{problem1}) is $C$-reversible,
in the sense that $C(i H u) = -i H(C(u))$ \cite[section V.1]{hairer06gni}. Moreover, since (\ref{sc2}) holds, then $C \circ S_h = S_h^{-1} \circ C$. In other words, the map
$S_h(u)$ is $C$-reversible \cite{hairer06gni} (or reversible for short). Notice that this also holds for palindromic compositions (\ref{split1}) with real coefficients.

In the sequel we will refer to compositions verifying (\ref{sc1}) as symmetric-conjugate or reversible methods.

Splitting and composition methods with complex coefficients have also interesting properties concerning the magnitude of the successive terms in the asymptotic expansion of the
local truncation error. Contrarily to methods with real coefficients, higher order error terms in the expansion of a given method have essentially a similar size as lower order terms
\cite{blanes22osc}. In addition, an integrator of a given order with the minimum number of flows typically achieves a good efficiency, whereas with real coefficients one has to
introduce additional parameters (and therefore more flows in the composition) for optimization purposes. It makes sense, then, to apply this class of schemes to equation (\ref{problem1}) and eventually compare their performance with splitting methods involving real coefficients, 
since in any case the presence of complex coefficients does not lead to an increment in the overall computational cost. 

The structure of the paper goes as follows. In section \ref{section2} we provide further experimental evidence of the preservation properties exhibited by $C$-reversible
splitting methods applied to different classes of matrices $H$ by considering several illustrative numerical examples. In section \ref{section3} we analyze in detail 
this type of methods and validate theoretically the observed results by stating two theorems concerning consistent reversible maps. Then, in section \ref{section4} we present
new symmetric-conjugate schemes up to order 6 specifically designed for the semi-discretized Schr\"odinger equation and other problems with the same
algebraic structure. Finally, these new methods are tested in section
\ref{section5} for a specific potential.

\section{Symmetric-conjugate splitting methods in practice: some illustrative examples}
\label{section2}

To illustrate the preservation properties exhibited by symmetric-conjugate (or reversible)
methods when applied to (\ref{problem1}) with $H = A + B$, we consider some low order compositions of this type. Specifically, the tests will be carried out with the following schemes:

\paragraph{Order 3.} The simplest symmetric-conjugate method corresponds to
\begin{equation} \label{split31}
   S_h^{[3,1]} = \e^{i h \overline{b}_0 B} \, \e^{i h \overline{a}_1 A} \, \e^{i h b_1 B} \, \e^{i h a_1 A} \,  \e^{i h b_0 B},
\end{equation}
with $a_1 = \frac{1}{2} + i \frac{\sqrt{3}}{6}$, $b_0 = \frac{a_1}{2}$, $b_1 = \frac{1}{2}$ and was first obtained in \cite{bandrauk91ies}. 
In addition, and as a representative of the schemes considered in section \ref{section4}, we
also use the following method,  with $a_j > 0$ and $b_j \in \mathbb{C}$, $\Re(b_j) > 0$:
\begin{equation} \label{split32}
   S_h^{[3,2]} = \e^{i h \overline{b}_0 B} \, \e^{i h a_1 A} \, \e^{i h \overline{b}_1 B} \, \e^{i h a_2 A} \,  \e^{i h b_1 B} \, \e^{i h a_1 A} \, \e^{i h b_0 B},
\end{equation}
where
\[
  a_1 = \frac{3}{10}, \quad a_2 = \frac{2}{5}, \quad b_0 = \frac{13}{126} - i \frac{\sqrt{59/2}}{63}, \quad b_1 = \frac{25}{63} + i \frac{5 \sqrt{59/2}}{126}.
\]

\paragraph{Order 4.}  The scheme has the same exponentials as (\ref{split32}),
\begin{equation} \label{split41}
   S_h^{[4]} = \e^{i h \overline{b}_0 B} \, \e^{i h \overline{a}_1 A} \, \e^{i h \overline{b}_1 B} \, \e^{i h a_2 A} \,  \e^{i h b_1 B} \, \e^{i h a_1 A} \, \e^{i h b_0 B},
\end{equation}
but now 
\[
  a_1 = \frac{1}{12} (3 + i \sqrt{15}), \quad a_2 = \frac{1}{2}, \quad b_0 = \frac{a_1}{2}, \quad b_1 = \frac{1}{24} (9 + i \sqrt{15}).
\]  
%By interchanging the role of $A$ and $B$ one gets another pair of equivalent integrators. 

When the matrix $H$ results from a space discretization of the time-dependent Schr\"odinger equation (for instance, by means of a pseudo-spectral method), 
then it is real 
symmetric and $A$, $B$ are also symmetric (in fact, $B$ is diagonal). It makes sense, then, start analyzing this situation, where, in addition, 
\emph{all the eigenvalues of $H$ are simple}.
To proceed, we generate a $N \times N$ real matrix with $N=10$ and uniformly distributed elements in the interval $(0,1)$, and take $H$ as its symmetric part. The symmetric matrix $A$ is generated analogously, and finally we fix $B = H - A$. Next we compute the approximations obtained by 
$S_h^{[3,1]}$, $S_h^{[3,2]}$ and $S_h^{[4]}$ for different values of $h$, 
determine their eigenvalues $\omega_j$ and compute the quantity 
\[
    D_h = \max_{1 \le j \le N} (\big| |\omega_j|-1 \big|)
\]
for each $h$. Finally, we depict $D_h$ as a function of $h$.

Figure \ref{figure1} (left) is representative of the results obtained in all cases we have tested: all $|\omega_j|$ are 1 (except round-off)
for some interval $0 < h < h^*$, and then there is always some
$\omega_{\ell}$ such that $|\omega_{\ell}| > 1$. In other words, $S_h^{[3,1]}$, $S_h^{[3,2]}$  
and $S_h^{[4]}$ behave as unitary maps in this interval. This is precisely what happens in the group
SU(2), as shown in \cite{blanes22asm}. 

The right panel of Figure \ref{figure1} is obtained in the same situation (i.e., $H$ real symmetric with simple eigenvalues), but now both $A$ and $B$ are no longer symmetric: essentially the same behavior as before is observed. Of course, when $h < h^*$, both the norm of $u$, $\mathcal{M}(u)$, and the expected value of the energy, $\mathcal{H}(u)$
are preserved for long times, as shown in \cite{blanes22asm}.

\begin{figure}[!ht] 
\centering
  \includegraphics[width=.49\textwidth]{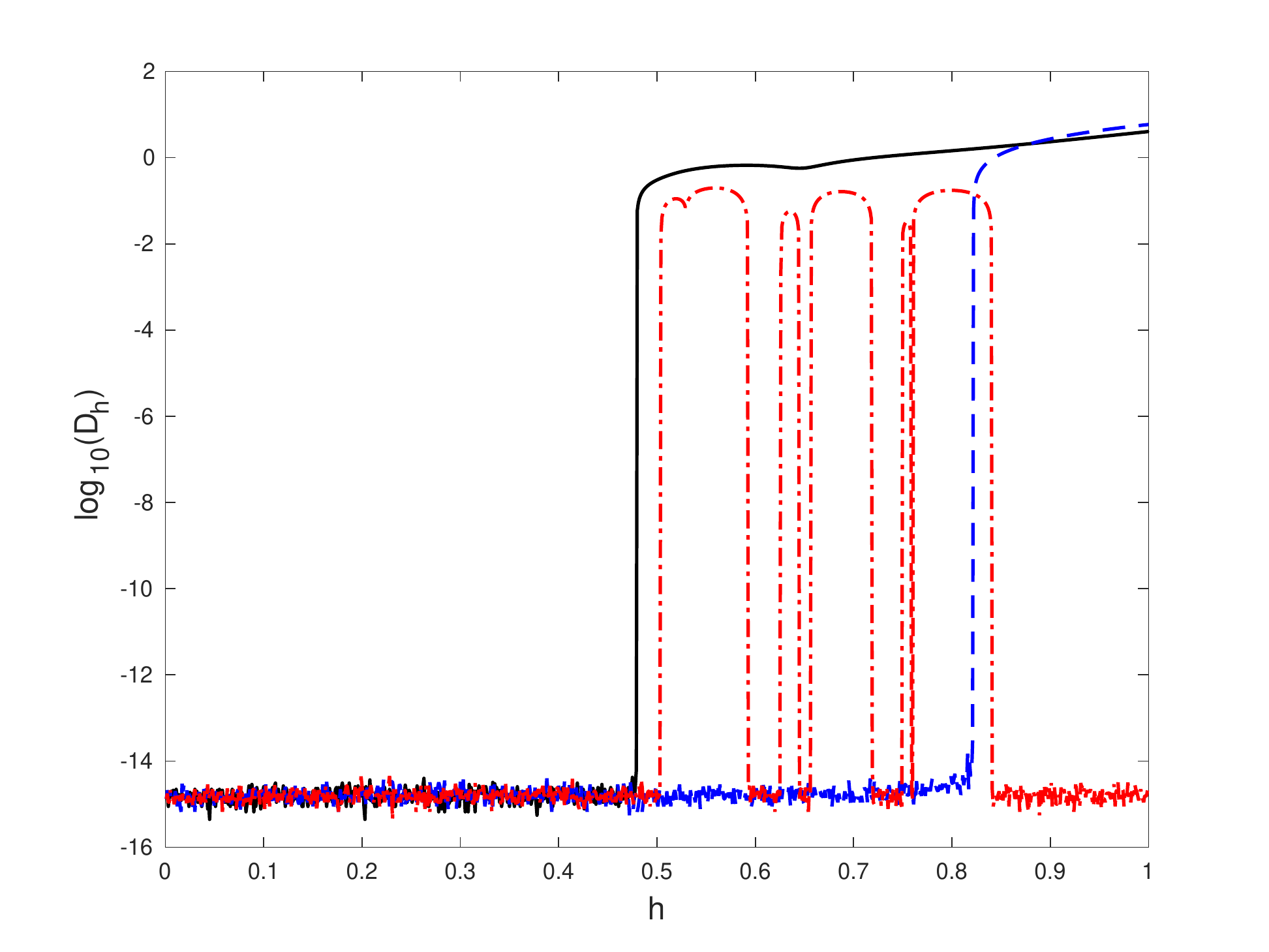}
  \includegraphics[width=.49\textwidth]{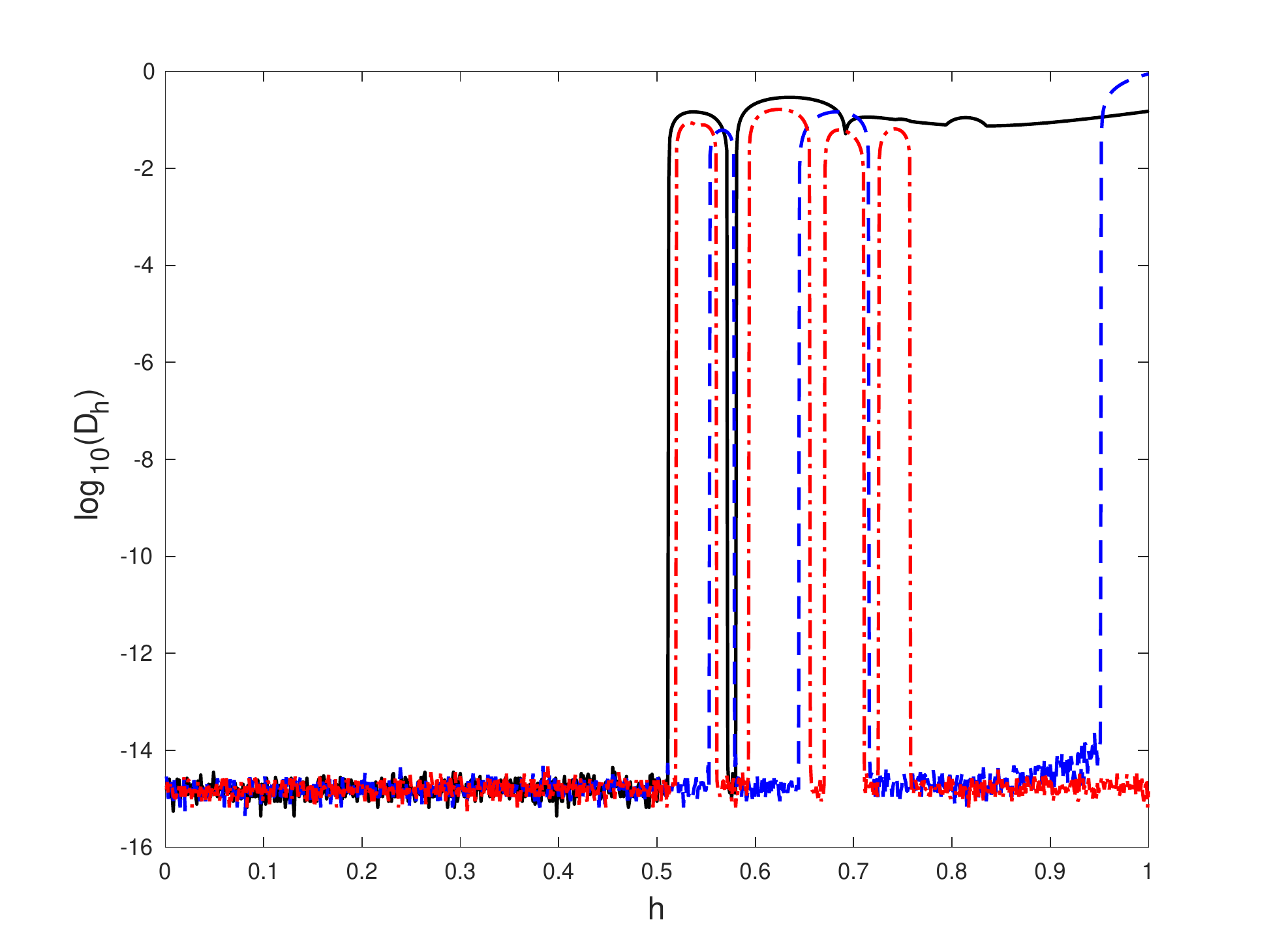}
\caption{\label{figure1} \small Absolute value of the largest eigenvalue of the approximations $S_h^{[3,1]}$ (black solid line), 
$S_h^{[3,2]}$ (red dash-dotted line) and $S_h^{[4]}$ (blue dashed line)
for different values of $h$ when
$H=A+B$ is a real symmetric matrix with simple eigenvalues. Left: $A$ and $B$ are also real symmetric. Right: $A$ and $B$ are real, but not symmetric.
}
\end{figure}

Our next simulation concerns a real (but not symmetric) matrix $H$ with all its eigenvalues \emph{real and simple}. Again, there exists a threshold $h^* > 0$ such that for 
$h < h^*$ the schemes
render unitary approximations. This is clearly visible in Figure \ref{figure2} (left panel). If we consider instead a completely 
arbitrary real matrix $H$, then the outcome is rather different:
$D_h > 0$ for any $h>0$ (right panel; for this example $D_h = 9.79 \cdot 10^{-4}$ already for $h=0.001$). 

\begin{figure}[!ht] 
\centering
  \includegraphics[width=.49\textwidth]{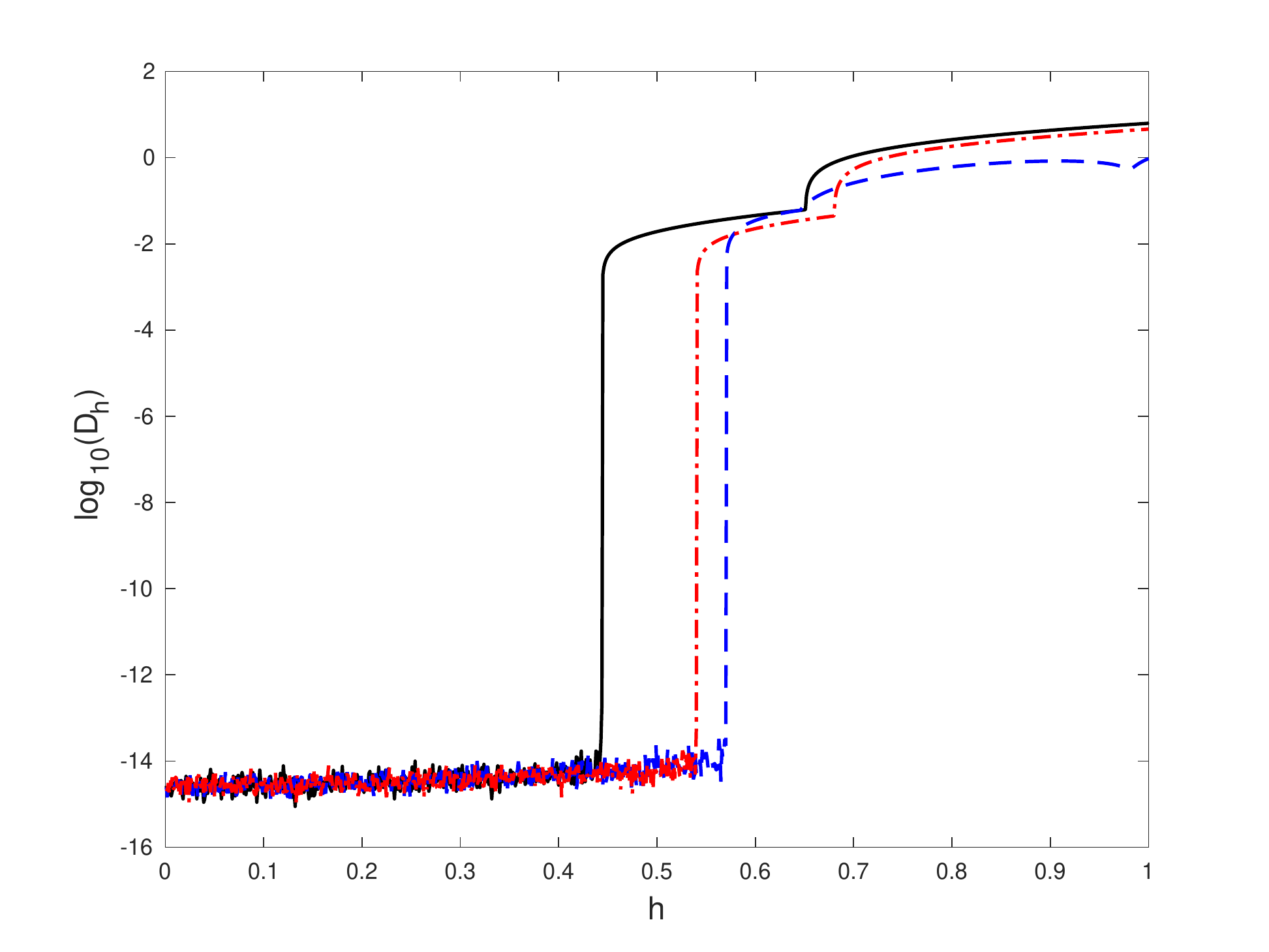}
  \includegraphics[width=.49\textwidth]{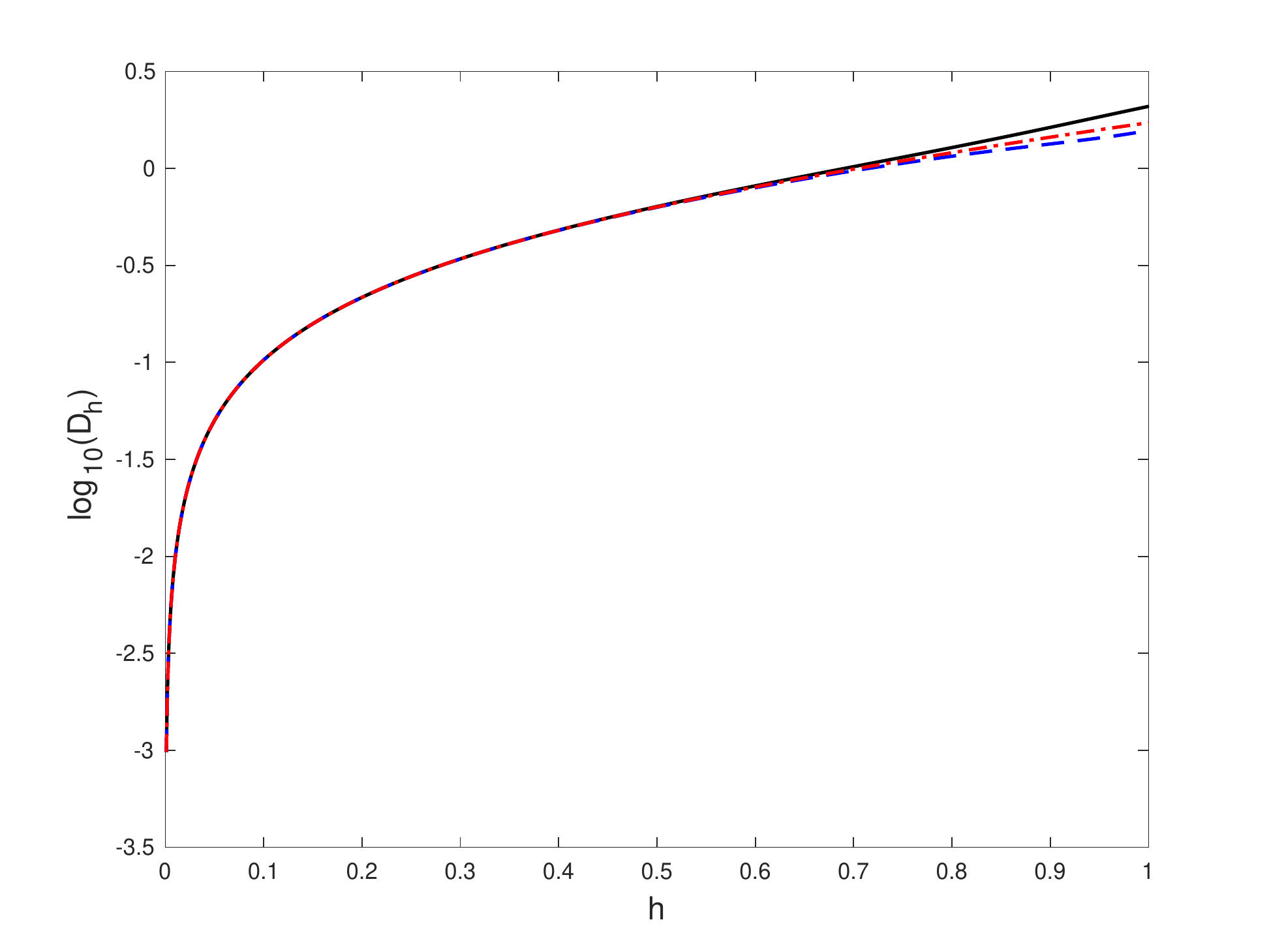}
\caption{\label{figure2} \small Same as Figure \ref{figure1} when
$H=A+B$ is a real (but not symmetric) matrix. Left: the eigenvalues of $H$ are real and simple. Right: the eigenvalues of $H$ are arbitrary.
}
\end{figure}

Next we illustrate the situation when \emph{the real matrix $H$ has multiple eigenvalues but is still diagonalizable}. As before, we consider first the analogue of Figure \ref{figure1}, namely:
$H$ is symmetric, with $A$ and $B$ symmetric matrices (Figure \ref{figure3}, left panel) and $A$ and $B$ are real, but not symmetric (right panel). In the first case we notice that,
whereas all the eigenvalues of the approximations rendered by $S_h^{[3,1]}$ and $S_h^{[4]}$ still have absolute value 1 for some interval $0 < h < h^*$, this is 
clearly not the case of $S_h^{[3,2]}$. If, on the other hand, the splitting is done is such a way that $A$ and $B$ are not symmetric (but still real), 
then $D_h > 0$ even for very small values of $h$. The same behavior is observed when $H$ is taken as a real (but not symmetric), diagonalizable matrix with multiple real eigenvalues.

\begin{figure}[!ht] 
\centering
  \includegraphics[width=.49\textwidth]{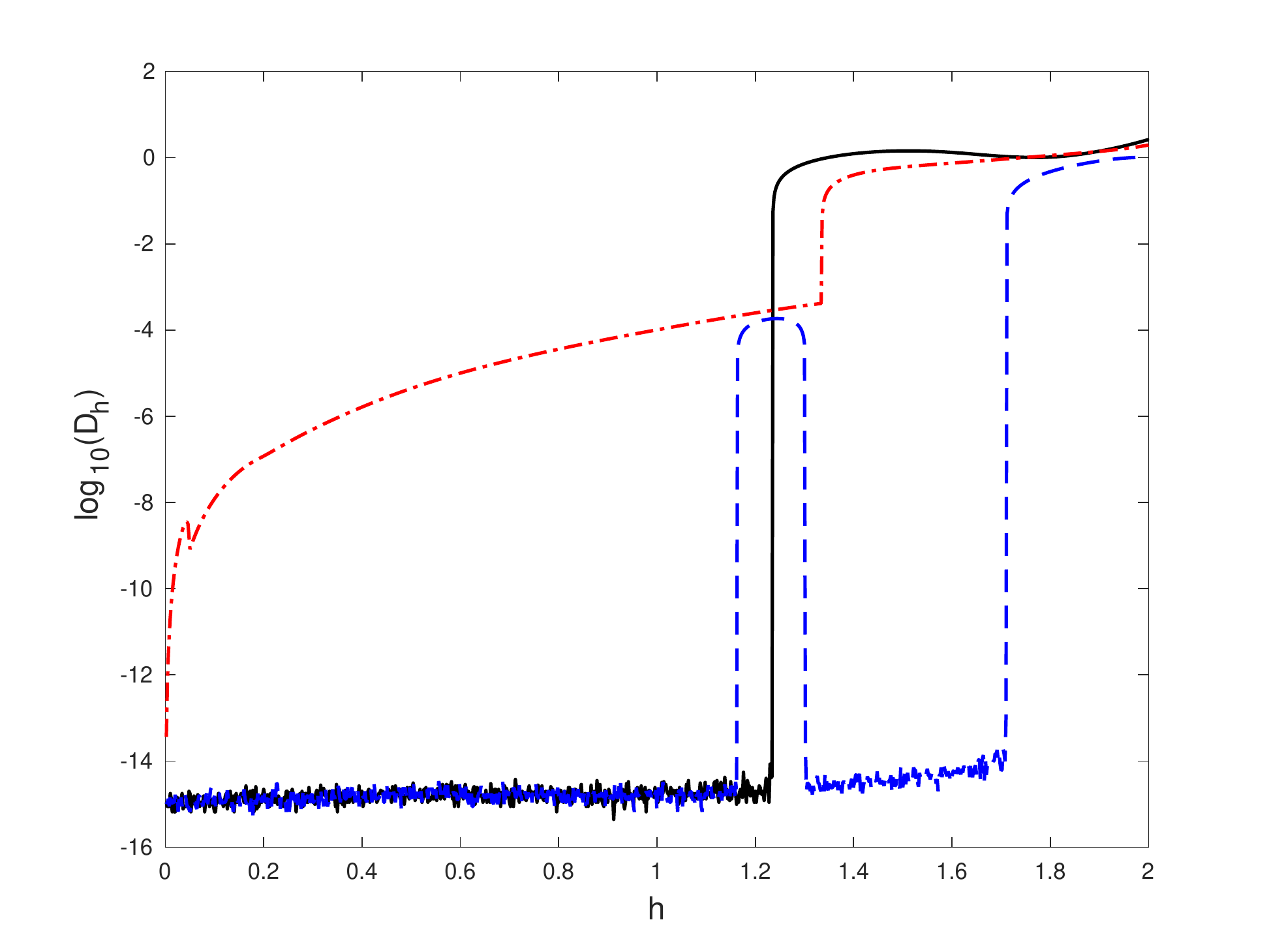}
  \includegraphics[width=.49\textwidth]{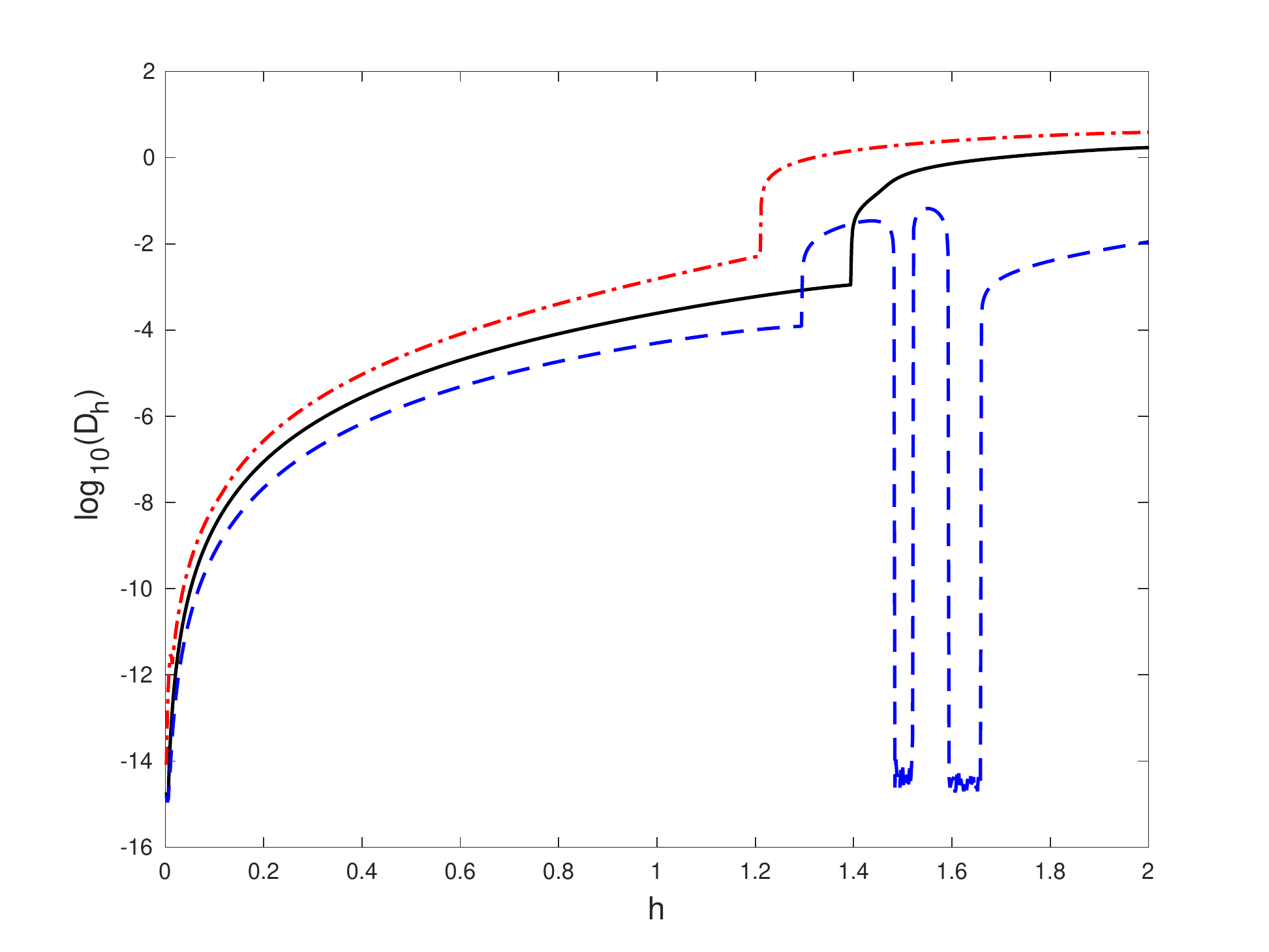}
\caption{\label{figure3} \small Same as Figure \ref{figure1} when
$H=A+B$ is a real symmetric matrix with multiple eigenvalues. Left: $A$ and $B$ are real symmetric matrices. Right: $A$ and $B$ are real, but
not symmetric.
}
\end{figure}

The different phenomena exhibited by these examples require then a detailed numerical analysis of the class of schemes involved, trying to explain in particular
the role played by the
eigenvalues of the matrix $H$ in the final outcome, as well as the different behavior of $S_h^{[3,1]}$ and $S_h^{[3,2]}$. This will be the subject of the
next section.

\section{Numerical analysis of reversible integration schemes}
\label{section3}

\subsection{Main results}
\label{sub3.1}

We next state two theorems and two additional corollaries that, generally speaking, justify the previous experiments and explain the good behavior exhibited by
reversible methods.

\begin{theorem}
\label{thm:simple}
Let  $H \in \mathbb{R}^{N\times N}$ be a real matrix and let $S_\ts \in \mathbb{C}^{N\times N}$ be a family of complex matrices depending smoothly on $\ts\in \mathbb{R}$ such that
\begin{itemize}
\item $S_\ts$ is a reversible map in the previous sense, so that
$$
 \overline{S}_\ts = S_\ts^{-1}; 
$$
\item $S_\ts$ is consistent with $\exp(i\ts H)$, i.e. there exists $p\geq 1$ such that
\begin{equation}
\label{eq:consist}
S_\ts \mathop{=}_{\ts\to 0} \e^{i\ts H} + \mathcal{O}(\ts^{p+1});
\end{equation}
\item the eigenvalues of $H$ are real and simple.
\end{itemize}
Then there exist 
\begin{itemize}
\item $D_\ts$,  a family of real diagonal matrices depending smoothly on $\ts$,
\item $P_\ts$, a family of real invertible matrices depending smoothly on $\ts$,
\end{itemize}
such that $P_\ts = P_0 + \mathcal{O}(\ts^p)$, $D_\ts = D_0 + \mathcal{O}(\ts^p)$  and, provided that $|\ts|$ is small enough, 
\begin{equation}
\label{eq:what_we_want}
S_\ts = P_\ts \, \e^{i\ts D_\ts} \, P_\ts^{-1}.
\end{equation}
\end{theorem}
%In this theorem, no special requirement for the matrices $P_0$ and $D_0$ other than $P_0^{-1} H P_0 = D_0$ is necessary.

\begin{corollary} 
\label{cor:simple} 
In the setting of Theorem \ref{thm:simple}, there exists a constant $C>0$ such that, provided that $|\ts|$ is small enough,  for all $u\in \mathbb{C}^N$ and all 
eigenvalues $\omega \in \sigma(H)$, one has
\begin{equation} \label{eq:cor_est_eigs}
 \sup_{n\geq 0} \, \Big||\Pi_\omega S_\ts^{n} u| - |\Pi_\omega  u| \Big| \leq C |\ts|^p |u|,
\end{equation}
where $\Pi_\omega$ denotes the spectral projector onto $\mathrm{Ker}(H-\omega I_N)$. Moreover, if $H$ is symmetric, the norm and the energy are almost conserved, 
in the sense that, for all $u\in \mathbb{C}^N$, it holds that 
\begin{equation} \label{eq:cor_mass_energ}
  \sup_{n \in \mathbb{Z}} \,  \big| \mathcal{M}(S_\ts^n u) - \mathcal{M}( u)  \big|\leq C |\ts|^p |u|^2  \qquad \mathrm{ and } \qquad 
  \sup_{n \in \mathbb{Z}} \, \big| \mathcal{H}(S_\ts^n u) - \mathcal{H}( u)  \big| \leq C |\ts|^p |u|^2, 
\end{equation}
where $\mathcal{M}( u) = |u|^2 $ and $\mathcal{H}( u) = \overline{u}^T H u $.
\end{corollary}

\begin{proof}[Proof of Corollary \ref{cor:simple}] First, we focus on \eqref{eq:cor_est_eigs}. We note that by consistency, we have
$$
 D_0= P_0^{-1} H P_0.
$$
Since the eigenvalues of $H$ are simple, it follows that the spectral projectors are all of the form
\begin{equation} \label{eq:spect_proj}
   \Pi^{(j)}= P_0 (e_j \otimes e_j) P_0^{-1}, 
\end{equation}
where $e_1,\ldots,e_N$ denotes the canonical basis of $\mathbb{R}^N$. Then, we note that for all $n\in \mathbb{Z}$, we have
$$
S_\ts^n = P_\ts \, \e^{in\ts D_\ts} \, P_\ts^{-1}.
$$
Therefore, since $\e^{in\ts D_\ts}$ is uniformly bounded with respect to $\ts$ and $n$ (because $D_\ts$ is a real diagonal matrix) and $P_\ts = P_0 + \mathcal{O}(\ts^p)$, it follows that
$$
S_\ts^n = P_0 \, \e^{in\ts D_\ts} \, P_0^{-1} + \mathcal{O}(\ts^p),
$$
where the implicit constant in $\mathcal{O}$ term does not depend on $n$ (here and later). Therefore, it is enough to use the explicit formula \eqref{eq:spect_proj} to prove that 
$$
\Pi^{(j)} S_\ts^n =  P_0 (e_j \otimes e_j) P_0^{-1}  P_0 \, \e^{in\ts D_\ts} P_0^{-1} + \mathcal{O}(\ts^p) = \e^{in\ts (D_\ts)_{j,j}} \Pi^{(j)} + \mathcal{O}(\ts^p).
$$
As a consequence, the estimate \eqref{eq:cor_est_eigs} follows directly by the triangular inequality :
$$
|\Pi^{(j)} S_\ts^{n} u| =  |\e^{in\ts (D_\ts)_{j,j}} \Pi^{(j)}  u +  \mathcal{O}(\ts^p)(u)|  \le  |\e^{in\ts (D_\ts)_{j,j}} \Pi^{(j)}  u | + |u| \mathcal{O}(\ts^p) = | \Pi^{(j)}  u |  + |u| \mathcal{O}(\ts^p) .
$$
Now, we focus on \eqref{eq:cor_mass_energ}. Here, since $H$ is assumed to be symmetric, its eigenspaces are orthogonal. Therefore by the Pythagorean theorem, we have
$$
\mathcal{M}(u) = \sum_{\omega \in \sigma(H)} |\Pi_\omega (u)|^2 \quad \mathrm{and} \quad \mathcal{H}(u) = \sum_{\omega \in \sigma(H)} \omega |\Pi_\omega (u)|^2.
$$
As a consequence, \eqref{eq:cor_mass_energ} follows directly of \eqref{eq:cor_est_eigs}.
\end{proof}

The main limitation of Theorem \ref{thm:simple} is the assumption on the simplicity of the eigenvalues of $H$. Indeed, even if this assumption is typically satisfied, it depends only on the equation we aim at solving and not of the numerical method one uses. The following theorem, which is a refinement of Theorem \ref{thm:simple}, remedies this point 
by making an assumption on the leading term of the consistency error (which is typically satisfied for generic choices of numerical integrators). 

\begin{theorem}
\label{thm:refin}
Let  $H \in \mathbb{R}^{N\times N}$ be a real matrix and let $S_\ts \in \mathbb{C}^{N\times N}$ be a family of complex matrices depending smoothly on $\ts$ such that
\begin{itemize}
\item $S_\ts$ is a reversible map, i.e.
$$
  \overline{S}_\ts = S_\ts^{-1}; 
$$
\item $S_\ts$ is consistent with $\exp(i\ts H)$, i.e.
\begin{equation}
\label{eq:consist_bis}
S_\ts \mathop{=}_{\ts\to 0} \e^{i\ts H} + i \ts^{p+1} R + \mathcal{O}(\ts^{p+2}),
\end{equation}
where $p\geq 1$ is the order of consistency and $R$ is a real matrix\footnote{The fact that $R$ is a real matrix is a consequence of the reversibility of $S_\ts$.};
\item $H$ is diagonalizable and its eigenvalues are real;
\item for all $\omega \in \sigma(H)$, the eigenvalues of $\Pi_\omega R_{| E_{\omega}(H)}$ are real and simple, where $\Pi_\omega$ denotes the spectral projector 
on $E_\omega(H) := \mathrm{Ker}(H-\omega I_N)$.
\end{itemize}
Then there exist 
\begin{itemize}
\item $D_\ts$,  a family of real diagonal matrices depending smoothly on $\ts$,
\item $P_\ts$, a family of real invertible matrices depending smoothly on $\ts$,
\end{itemize}
such that, both $P_0^{-1} B P_0$ and $P_0^{-1} H P_0$ are diagonal, where $B := \sum_{\omega \in \sigma(H)} \Pi_\omega \, R \, \Pi_\omega$, and provided that $|\ts|$ is small enough, it holds that
\begin{equation}
\label{eq:what_we_want_bis}
S_\ts = P_\ts \, \e^{i\ts D_\ts} \, P_\ts^{-1}.
\end{equation}
\end{theorem}

\begin{corollary} \label{cor:refin}
 In the setting of Theorem \ref{thm:refin}, there exists a constant $C>0$ such that, provided that $|\ts|$ is small enough,  for all $u\in \mathbb{C}^N$, 
 all $\omega \in \sigma(H)$ and all $\lambda \in \sigma(\Pi_\omega R_{| E_{\omega}(H)})$, we have
$$
 \sup_{n\geq 0} \, \Big||\mathcal{P}_{\lambda,\omega} S_\ts^{n} u| - |\mathcal{P}_{\lambda,\omega}  u| \Big| \leq C |\ts| |u|,
$$
 where $\mathcal{P}_{\lambda,\omega}$ denotes the projector along $\bigoplus_{(\eta,\mu)\neq(\lambda,\omega)} E_\eta(\Pi_\mu R_{| E_{\mu}(H)})$ 
 onto $E_\lambda(\Pi_\omega R_{| E_{\omega}(H)})$. 
 
Moreover, if $H$ and $R$ are symmetric, for all $\omega \in \sigma(H)$, one gets
$$
 \sup_{n\geq 0} \, \Big||\Pi_\omega S_\ts^{n} u|^2 - |\Pi_\omega  u|^2 \Big| \leq C |\ts| |u|^2,
$$
and the mass and the energy are almost conserved, i.e. for all $u\in \mathbb{C}^N$, it holds that
\[
  \sup_{n \in \mathbb{Z}} \,  \big| \mathcal{M}(S_\ts^n u) - \mathcal{M}( u)  \big|\leq C |\ts| |u|^2  \qquad \mathrm{ and } \qquad 
  \sup_{n \in \mathbb{Z}} \, \big| \mathcal{H}(S_\ts^n u) - \mathcal{H}( u)  \big| \leq C |\ts| |u|^2, 
\]
where, as before, $\mathcal{M}( u) = |u|^2 $ and $\mathcal{H}( u) = \overline{u}^T H u $.
\end{corollary}

\begin{proof}[Proof of Corollary \ref{cor:refin}] The proof is almost identical to the one of Corollary \ref{cor:simple}. The key point is that, since 
both $P_0^{-1} B P_0$ and $P_0^{-1} H P_0$ are diagonal, then the projectors $\mathcal{P}_{\lambda,\omega}$ are exactly 
the projectors $\Pi^{(j)}$, $1\leq j \leq N$ (given by \eqref{eq:spect_proj}). Note that, contrary to Theorem \ref{thm:simple}, in Theorem \ref{thm:refin} one does not claim that $P_\ts = P_0 + \mathcal{O}(\ts^p)$. A priori, here, in general, the best estimate we expect is $P_\ts = P_0 + \mathcal{O}(\ts)$ (which follows directly from the smoothness of $P_\ts$  with respect to $\ts$). It is this loss which explains why, in Corollary \ref{cor:refin}, the error terms are of order $\mathcal{O}(h)$ whereas they are of order $\mathcal{O}(h^p)$ in Corollary \ref{cor:simple}. 
\end{proof}

\paragraph{Remark.} Before starting the proof of these theorems, let us provide some comments about the context and the ideas involved.
\begin{itemize}
\item In Theorem \ref{thm:simple} and its proof, we are just putting $S_h$ in Birkhoff normal form. The fact that $S_h$ can be diagonalized is due to the simplicity of the eigenvalues of $H$ while the fact that its eigenvalues are complex numbers of modulus $1$ is due the reversibility of $S_h$. This approach is robust and well known, in particular it can be extended to the nonlinear setting (see e.g.  \cite[section V.1]{hairer06gni}). Note that here, we reach convergence of the Birkhoff normal form because the system is linear.

\item Theorem \ref{thm:refin} is a refinement of Theorem \ref{thm:simple}. To prove the absence of resonances due to the multiplicity of the eigenvalues of $H$, we use the first correction to the frequencies generated by the perturbation of $H$ (i.e., the projections of $R$ in  Theorem \ref{thm:refin}). This approach is typical of 
what one does in the proof of Nekhoroshev theorems or KAM theorems (see also  \cite{hairer06gni}).

\item In order to give some intuition about the proof and the assumptions of Theorem \ref{thm:simple}, let us prove simply that, provided $h$ is small enough, $S_h$ is conjugated to a unitary matrix. Indeed, since $S_h$ is reversible it writes as
$$
S_h = e^{ihH_h},
$$
where $H_h = H + \mathcal{O}(h^p)$ is a real matrix (provided that $h$ is small enough). Now, since the set of the real matrices whose eigenvalues are simple and real is open in the space of the real matrices (by continuity of the eigenvalues)  and $H_h$ is a real perturbation of such a matrix ($H$ by assumption), we deduce that, provided $h$ is small enough, its eigenvalues are simple and real. This implies that $H_h$ is conjugated to a real diagonal matrix and so that $S_h$ is conjugated to a unitary matrix.
\end{itemize}

\subsection{Technical lemmas}
\label{sub3.2}

In the proof of the previous theorems we will make use of the following three lemmas.

\begin{lemma} \label{lem:conj} 
Let $M$ be a complex matrix and let $P$ be a complex invertible matrix. Then  $\mathrm{ad}_{P^{-1} M P}$ and $\mathrm{ad}_{M}$ are similar. More precisely, 
$$
\mathrm{ad}_{\mathrm{int}_P \, M} = (\mathrm{int}_P)  \mathrm{ad}_{M} (\mathrm{int}_P)^{-1},
$$
where $\mathrm{int}_P M := P^{-1} M P$. Here $\mathrm{ad}_M$ stands for the adjoint operator: $\mathrm{ad}_M X := [M, X] = M X - X M$, for any matrix $X$.
\end{lemma}
\begin{proof}
A straightforward calculation shows that, for any $X$,
\[
  (\mathrm{int}_P) \mathrm{ad}_{M}  X = P^{-1} [M, X] P = [P^{-1} M P, P^{-1} X  P] = \mathrm{ad}_{\mathrm{int}_P \, M} (P^{-1} X P) = 
  \mathrm{ad}_{\mathrm{int}_P \, M} (\mathrm{int}_P) X.
\]  
\end{proof}

\begin{lemma} \label{lem:sup} 
Let $M$ be a complex matrix. Then $M$ is diagonalizable if and only if the kernel and the image of $\mathrm{ad}_M$ are supplementary, i.e.
\begin{equation}
\label{eq:cap}
\mathrm{Ker}_{\mathbb{C}} \ \mathrm{ad}_M \cap \mathrm{Im}_{\mathbb{C}} \ \mathrm{ad}_M = \{0\}.
\end{equation}
\end{lemma}
\begin{proof} 
We can assume, in virtue of Lemma  \ref{lem:conj} and without loss of generality, that 
$M$ is in Jordan normal form\footnote{Indeed, the property \eqref{eq:cap} is clearly invariant by conjugation of $\mathrm{ad}_M$ and by 
Lemma \ref{lem:conj} we know that $\mathrm{ad}_M$ is conjugated to the adjoint representation of any Jordan normal form of $M$.}.
On the one hand, if $M$ is diagonal, we have $\mathrm{ad}_M A = (( m_{i,i} - m_{j,j}) A)_{i,j}$ and so the support of the matrices in 
$\mathrm{Ker}_{\mathbb{C}} \ \mathrm{ad}_M$ and $\mathrm{Im}_{\mathbb{C}} \ \mathrm{ad}_M$ are clearly disjoint (which implies \eqref{eq:cap}). 
Conversely, doing calculations by blocks it is enough to consider the case where $M= \lambda I_N + \mathcal{N}$ is a Jordan matrix 
(i.e. $\lambda \in \mathbb{C}$ and $\mathcal{N}$ nilpotent). Then we just have to note that $\mathrm{ad}_{\lambda I_N + \mathcal{N}} = \mathrm{ad}_{\mathcal{N}}$ 
and that since $\mathrm{ad}_{\mathcal{N}}$ is nilpotent necessarily we have 
$\mathrm{Ker}_{\mathbb{C}} \ \mathrm{ad}_{\mathcal{N}} \cap \mathrm{Im}_{\mathbb{C}} \ \mathrm{ad}_{\mathcal{N}} \neq \{0\}$.
\end{proof}

\begin{lemma}
\label{lem:diag_cont}
Let $M_\ts$ be a family of real matrices depending smoothly on $\ts$ and of the form 
$$
M_\ts = M_0 + \mathcal{O}(\ts^p), \quad \mbox{ where } \quad p\geq 1.
$$
 If $M_0$ is diagonalizable on $\mathbb{C}$, then there exists a family of real matrices $\chi_\ts$, depending smoothly on $\ts$, such that if $|\ts|$ is small enough,
  $\e^{-\ts^p \chi_\ts} M_\ts \, \e^{\ts^p\chi_\ts}$ commutes with $M_0$, i.e.
$$
[\e^{\ts^p\chi_\ts} M_\ts \, \e^{-\ts^p\chi_\ts} , M_0 ] =0.
$$
\end{lemma}

\begin{proof} We aim at designing the family $\chi_\ts$ as solution of the equation
$$
\mathrm{ad}_{M_0} \left( \e^{\ts^p\chi_\ts} M_\ts \, \e^{-\ts^p\chi_\ts} \right) =0.
$$
Thanks to the well known identity $\e^{A} B \, \e^{-A} = \e^{\mathrm{ad}_A} B$, this equation rewrites as
\begin{equation} \label{eq:v1}
 \mathrm{ad}_{M_0} \left( \e^{\ts^p \mathrm{ad}_{\chi_\ts}}  M_\ts \right) = 0.
\end{equation}
Next we write the Taylor expansion of $M_\ts$ at order $p$ as
$$
M_\ts = M_0 + \ts^p R_\ts,
$$
where $R_\ts$ is a family of real matrices depending smoothly on $\ts$. Then, isolating the terms of order $0$ (and dividing by $\ts^p$), the equation \eqref{eq:v1} leads to
$$
f(\ts,\chi_\ts) := \mathrm{ad}_{M_0} \left(   \e^{\ts^p \mathrm{ad}_{\chi_\ts}}  R_\ts-\varphi_1(\ts^p \mathrm{ad}_{\chi_t}) \, \mathrm{ad}_{M_0} \chi_\ts  \right) = 0,
$$
where
$\varphi_1(z) := \frac{e^z - 1}z$. We restrict ourselves to  $\chi_\ts$ in $\mathrm{Im}_{\mathbb{R}} \, \mathrm{ad}_{M_0}$ and consider $f$ as a 
smooth map from $\mathbb{R} \times \mathrm{Im}_{\mathbb{R}} \, \mathrm{ad}_{M_0}$ to $\mathrm{Im}_{\mathbb{R}} \, \mathrm{ad}_{M_0}$. To solve the equation $f(\ts,\chi_\ts)=0$ using the implicit function theorem, we just have to design $\chi_0$ so that
$$
f(0,\chi_0) =  \mathrm{ad}_{M_0}    R_0- \mathrm{ad}_{M_0} \chi_0  = 0
$$
and prove that $\mathrm{d}_\chi f(0,\chi_0) = - \mathrm{ad}_{M_0} : \mathrm{Im}_{\mathbb{R}} \, \mathrm{ad}_{M_0} \to \mathrm{Im}_{\mathbb{R}} \, \mathrm{ad}_{M_0}$ is invertible. Actually, these properties are clear because the first one is a consequence of the second one, whereas the second follows directly from Lemma \ref{lem:sup}.
\end{proof}

\subsection{Proofs of the theorems}
\label{sub3.3}

We are now in a position to prove Theorems \ref{thm:simple} and  \ref{thm:refin}. Without loss of generality, and to simplify notations, we assume that $H$ is diagonal
$$
H = \begin{pmatrix} \omega_1 I_{n_1} \\ & \ddots \\ & & \omega_d I_{n_d} \end{pmatrix},
$$
where $\omega_1 <\cdots < \omega_d$ denote the eigenvalues of $H$ and $n_1,\cdots,n_d$ are positive integers satisfying $n_1+\cdots+n_d=N$.

Thanks to the consistency assumption \eqref{eq:consist_bis} (which is equivalent to \eqref{eq:consist}), provided that $|\ts|$ is small enough, $S_\ts$ rewrites as
$$
S_\ts = \e^{i \ts H_\ts}, \quad \mbox{ where } \quad H_\ts = H+\ts^p R + \mathcal{O}(\ts^{p+1}). 
$$
Moreover, the reversibility assumption $S_\ts^{-1} = \overline{S}_\ts $ implies that $H_\ts$ is a real matrix (provided that $|\ts|$ is small enough). Note that, hence, 
we deduce that $R$ is also a real matrix. Then, applying Lemma \ref{lem:diag_cont} to $H_h$, we get a family of real matrices $\chi_\ts$ such that, 
provided that $|\ts|$ is small enough,
$$
[W_\ts , H ] =0, \qquad \mbox{ where } \qquad W_\ts = \e^{\ts^p\chi_\ts} H_\ts \, \e^{-\ts^p\chi_\ts}.
$$
We conclude that $W_\ts$ is block-diagonal (with the same structure of blocks as $H$), i.e. there exists some $n_j \times n_j$ real matrices $W_\ts^{(j)}$ such that
\begin{equation} \label{wh}
W_\ts = \begin{pmatrix} W_\ts^{(1)} \\ & \ddots \\ & & W_\ts^{(d)} \end{pmatrix}.
\end{equation}
As a consequence, if the eigenvalues of $H$ are simple (i.e. $d=N$ and $n_j=1$ for all $j$) then $W_\ts$ is diagonal. Therefore, in this case, it is enough to set 
$P_\ts=\e^{-\ts^p\chi_\ts}$ and $W_\ts = D_\ts$ to conclude the proof of Theorem \ref{thm:simple}.

\bigskip

So, from now on, we only focus on the proof of Theorem \ref{thm:refin}. First, we aim at identifying the matrices on the blocks in (\ref{wh}). The Taylor expansion of $W_\ts$ is
clearly
$$
W_\ts = H +  \ts^{p} B + \mathcal{O}(\ts^{p+1}),  \qquad \mbox{ with } \qquad B := R+ [\chi_0,H].
$$
However, since $[W_\ts,H]=0$, we deduce that $[B,H]=0$ and so that $B$ is block-diagonal. Moreover, since the matrix $[\chi_0,H]$ is identically equal to zero 
on the diagonal blocks, the diagonal blocks of $B$ are exactly those of $R$. As a consequence, with a slight abuse of notations, we may write
$$
W_\ts^{(j)} = \omega I_{n_j} + \ts^p B^{(j)} + \ts^{p+1}Y_\ts^{(j)}, \qquad \mbox{ where } \qquad  B^{(j)} := \Pi_{\omega_{j}} R_{| E_{\omega_j}(H)}
$$
and $Y_\ts^{(j)}$ is a family of real matrices depending smoothly on $\ts$. 

Next we aim at diagonalizing these blocks. By assumption, the eigenvalues of each matrix $B^{(j)}$ are real and simple. Therefore, all  $B^{(j)}$ are diagonalizable. 
As a consequence, and again by applying Lemma \ref{lem:diag_cont}, we get a family of real matrices $\Upsilon^{(j)}_\ts$ such that if $|\ts|$ is small enough, for all $j\in \llbracket 1,d\rrbracket$ we have
$$
\Big[ \e^{\ts \Upsilon^{(j)}_\ts} (B^{(j)} + \ts Y_\ts^{(j)})  \e^{-\ts \Upsilon^{(j)}_\ts }, B^{(j)} \Big] = 0.
$$
This means that the eigenspaces of $B^{(j)}$ are stable by the action of $ \e^{\ts \Upsilon^{(j)}_\ts} (B^{(j)} + \ts Y_\ts^{(j)})  \e^{-\ts \Upsilon^{(j)}_\ts }$. However, by assumption, these spaces are lines. Therefore, if $Q^{(j)}$ is a real invertible matrix such that $ Q^{(j)} B^{(j)}   (Q^{(j)})^{-1}$ is diagonal then 
$ Q^{(j)} \e^{\ts \Upsilon^{(j)}_\ts} (B^{(j)} + \ts Y_\ts^{(j)})  \e^{-\ts \Upsilon^{(j)}_\ts }(Q^{(j)})^{-1} $ is also diagonal.

Finally, as a consequence, setting
$$
P_\ts := \e^{-\ts^p\chi_\ts} \begin{pmatrix}  \e^{-\ts \Upsilon^{(1)}_\ts} Q^{(1)} \\ & \ddots \\ & &  \e^{-\ts \Upsilon^{(d)}_\ts} Q^{(d)}  \end{pmatrix} 
$$
we have proven that $D_\ts := P_\ts^{-1} H_\ts P_{\ts}$ is real diagonal, which concludes the proof of Theorem \ref{thm:refin}.

\subsection{Applications to reversible splitting and composition methods}
\label{sub3.4}

Theorems \ref{thm:simple} and  \ref{thm:refin} shed light on the behavior observed in the examples collected in Section \ref{section2}. Thus, suppose 
$H = A + B$ is a real symmetric matrix, with $A$, $B$ also real. Furthermore, consider a splitting scheme $S_h$ of the form (\ref{split1}) 
with coefficients satisfying the symmetry conditions (\ref{sc1}) and consistency,
$$
a_0 + \cdots + a_{2n} = 1, \qquad\qquad b_0 + \cdots + b_{2n-1} =1.
$$ 
Clearly, $S_h$ is a reversible map and moreover, it is consistent with $\e^{i\ts H}$ at least at order $1$, so that (\ref{eq:consist}) holds with $p \ge 1$. Since $H$
is real symmetric, it is diagonalizable. Therefore, if the eigenvalues of $H$ are simple, the dynamics of $(S_\ts^n)_{n\in \mathbb{Z}}$ is given by Theorem \ref{thm:simple}:
for sufficiently small $h$, there exist real matrices $D_h$ (diagonal) and $P_h$ (invertible) so that  $S_h^n = P_h \, \e^{i n D_h} P_h^{-1}$, all the eigenvalues of $S_h$
verify $|\omega_j| = 1$ and $\mathcal{M}(u)$ and $\mathcal{H}(u)$ are almost preserved for long times. This corresponds to the examples 
of Figure \ref{figure1}. The same
conclusions apply as long as $H$ is a real matrix with all its eigenvalues real and simple (Figure \ref{figure2}, left), whereas the general case of complex eigenvalues is
not covered by the theorem, and no preservation is ensured (Figure \ref{figure2}, right).

Suppose now that the real matrix $H$ has multiple real eigenvalues, but is still diagonalizable, and that $A$ and $B$ are real and symmetric. In that case, a symmetric-conjugate
splitting method satisfy both conditions (\ref{sc2}) and (\ref{sc3}), so that it can be written as 
$$
S_\ts = e^{i\ts H_\ts},
$$
where $H_\ts$ is a family of real matrices whose even terms in $h$ are symmetric and odd terms are skew-symmetric. Suppose in addition that $S_h$ is of even order
(i.e., $p$ is even in (\ref{eq:consist_bis})). In that case the matrix $R$ in Theorem \ref{thm:refin} is symmetric, and so its eigenvalues are real.
Moreover, since $R$ strongly depends on the coefficients $a_j, b_j$ and the decomposition $H=A+B$, it is very likely that typically the eigenvalues of the operators $\Pi_\omega R_{| E_{\omega}(H)}$ are simple and so that the dynamics of $(S_\ts^n)_{n\in \mathbb{Z}}$ is given by Theorem \ref{thm:refin} and is therefore similar to the one of
$(\e^{in\ts H})_{n\in \mathbb{Z}}$. Notice that this does not necessarily hold if the scheme is of odd order and/or $A$ and $B$ are not symmetric. This phenomenon is
clearly illustrated in the examples of Figure \ref{figure3} by methods $S_h^{[3,2]}$ and $S_h^{[4]}$.

Notice, however, that method $S_h^{[3,1]}$, although of odd order, 
works in fact better than expected from the previous considerations. The reason for this behavior resides in the following 

\begin{proposition}
The 3th-order symmetric-conjugate splitting method
\[
   S_h^{[3,1]} = \e^{i h \overline{b}_0 B} \, \e^{i h \overline{a}_1 A} \, \e^{i h b_1 B} \, \e^{i h a_1 A} \,  \e^{i h b_0 B},
\]
with $a_1 = \frac{1}{2} + i \frac{\sqrt{3}}{6}$, $b_0 = \frac{a_1}{2}$, $b_1 = \frac{1}{2}$, is indeed conjugate to a reversible integrator $V_h$
of order 4, i.e., there exists a real near-identity transformation $F_h$ such that $F_h \, S_h^{[3,1]} \, F_h^{-1} = V_h = \e^{i h H} + \mathcal{O}(h^5)$ and 
$\overline{V}_h = V_h^{-1} $.
\end{proposition}

\begin{proof}
Method $S_h^{[3,1]}$ constitutes in fact a particular case of a composition $\psi_h = \mathcal{S}_{\bar{\alpha} h}^{[2]} \, \mathcal{S}_{\alpha h}^{[2]}$, where
$\mathcal{S}_{h}^{[2]}$ is a time-symmetric 2nd-order method and $\alpha = a_1$. Specifically, $S_h^{[3,1]}$ is recovered when
$\mathcal{S}_{h}^{[2]} = \e^{\frac{h}{2} B} \, \e^{h A} \, \e^{\frac{h}{2} B}$.
Therefore, it can be written as
\[
   \mathcal{S}_{h}^{[2]} = \exp( i h H - i h^3 F_3 + i h^5 F_5 + \cdots)
\]
for certain real matrices $F_{2j+1}$. In consequence, by applying the BCH formula, one gets $\psi_h = \e^{W(h)}$, with
\[
     W(h) = i h H + \frac{1}{2} h^4 |\alpha|^2(\alpha^2-\bar{\alpha}^2) [H, F_3] + i h^5 \big( w_{5,1} F_5 + w_{5,2} [H,[H, F_3]] \big) + \mathcal{O}(h^6).
\]
Here $w_{5,j}$ are polynomials in $\alpha$. Now let us consider
\[
  V_h = \e^{V(h)} = \e^{\lambda h^3 F_3} \, \e^{W(h)} \,   \e^{-\lambda h^3 F_3} 
\]
for a given parameter $\lambda$. Then, clearly,
\[
  V(h) = \e^{\lambda h^3 \mathrm{ad}_{F_3}} W(h) = i h H + h^4 \left( \frac{1}{2} \alpha^3 - i \lambda \right) [H, F_3] + \mathcal{O}(h^5),
\]
so that by choosing $\lambda = -\frac{i}{2} \alpha^3  = - \frac{\sqrt{3}}{18}$, we have $V(h) = i h H  + \mathcal{O}(h^5)$ and the stated result is obtained, 
with $F_h =  \e^{\lambda h^3 F_3} $.    
\end{proof}

This result can be generalized as follows: given a time-symmetric method $\mathcal{S}_{h}^{[2k]}$ of order $2k$, if $\alpha$ is chosen so that
the composition $\psi_h = \mathcal{S}_{\bar{\alpha} h}^{[2k]} \, \mathcal{S}_{\alpha h}^{[2k]}$ is of order $2k+1$, then $\psi_h$ is conjugate to a 
reversible method of order
$2k+2$.

\

Theorems \ref{thm:simple} and  \ref{thm:refin} also allow one to explain the good behavior shown by symmetric-conjugate composition methods
for this type of problems.
In fact, suppose $H$ is a real symmetric matrix and $\Phi_H^z$ is a family of linear maps which are consistent with $\e^{izH}$ at least at order $1$ and satisfy
$$
(\Phi_H^z)^{-1} = \overline{\Phi_H^{\overline{z}}}.
$$ 
If we define $S_\ts$ as  the symmetric-conjugate composition 
$$
S_\ts = \Phi_H^{\alpha_0 \ts} \cdots \Phi_H^{\alpha_n \ts}, 
$$
where $\alpha_j$ are some complex coefficients satisfying the symmetry condition
$$
\alpha_{n-j} = \overline{\alpha}_j, \qquad j=1,2,\ldots
$$
and the consistency condition
$$
\alpha_0 + \cdots + \alpha_n = 1,
$$
then $S_\ts$ is a reversible map. Moreover, it is consistent with $\e^{i\ts H}$ at least at order $1$. Therefore, one can apply 
Theorem \ref{thm:simple} and Theorem \ref{thm:refin} also in this case. Notice, in particular, that even if the maps $\e^{i h a_j A}$ and/or $\e^{i h b_j B}$
in the symmetric-conjugate splitting method (\ref{split1}) are not computed exactly, but only conveniently approximated (for instance, by the midpoint rule), the previous theorems
still apply, so that one can expect good long term behavior from the resulting approximation.

\section{Symmetric-conjugate splitting methods for the Schr\"odinger equation}
\label{section4}

%\subsection{Analytical framework}
%\label{subsec4.1}

An important application of the previous results corresponds to
the numerical integration of the time dependent Schr\"odinger equation ($\hbar = m = 1$)
\begin{equation}   \label{Schr0}
  i \frac{\partial}{\partial t} \psi (x,t)   = \hat{H} \psi (x,t), \qquad\quad \psi(x,0)=\psi_0(x),
\end{equation}
where $\psi: \mathbb{R}^3 \times \mathbb{R} \longrightarrow \mathbb{C}$. The Hamiltonian operator $\hat{H}$ is the sum
$\hat{H} = \hat{T} + \hat{V}$
of the kinetic energy operator $\hat{T}$ and the potential $\hat{V}$. Specifically,
\[
  (\hat{T} \psi)(x) = -\frac{1}{2} \Delta \psi(x,t), \qquad\quad (\hat{V} \psi)(x) = \hat{V}(x) \psi(x,t).
\]
In addition, a simple computation shows that $[\hat{V},[\hat{T},\hat{V}]] \ \psi  = | \nabla \hat{V}|^2 \psi$, and therefore 
\begin{equation} \label{prop.1}
 [\hat{V},[\hat{V},[\hat{V},\hat{T}]]] \ \psi = 0.
\end{equation}  
Assuming $d=1$ and 
periodic boundary conditions, the application of a pseudo-spectral
method in space (with $N$ points) leads to the $N$-dimensional system (\ref{problem1}), where $u(0) = u_{0} \in \mathbb{C}^N$
and $H$ represents the (real symmetric) $N \times N$ matrix associated with the operator $-\hat{H}$ \cite{lubich08fqt}. Now
\[ 
   H = A + B,
\]
where $A$ is the (minus) differentiation matrix corresponding to the discretization of $\hat{T}$ (a real and symmetric matrix) and $B$ is the diagonal
matrix associated to $-\hat{V}$ at the grid points. Since  $\exp(t A)$ can be efficiently computed with the fast Fourier transform (FFT)
algorithm, it is a common practice to use splitting methods of the form (\ref{split1}) to integrate this problem. In this respect, notice that property
(\ref{prop.1}) will be inherited by the matrices $A$ and $B$ only if the number of discretization points $N$ is sufficiently large to achieve spectral
accuracy, i.e.,
\begin{equation} \label{rkn}
   [B,[B,[B,A]]] u= 0 \qquad \mbox{ if $N$ is large enough. }
\end{equation}
Assuming this is satisfied, then  there is a reduction in the number of conditions necessary to construct a method (\ref{split1}) of a given order $p$ \cite{hairer06gni,blanes16aci}.
Integrators of this class are
 sometimes
called Runge--Kutta--Nystr\"om (RKN) splitting methods \cite{blanes22rkn}. 

Two further points are worth remarking. First, 
the computational cost of evaluating (\ref{split1}) is not significantly increased by incorporating complex coefficients into the scheme, since
one has to use complex arithmetic anyway. Second, since $\sum_j a_j = 1$ for a consistent method, if $a_j \in \mathbb{C}$, then 
both positive \emph{and} negative imaginary parts are present, and this can lead to severe instabilities due to the unboundedness of the
Laplace operator \cite{castella09smw,hansen09hos}. On the other hand, the spurious effects introduced by complex $b_j$ can be eliminated (at least for sufficiently small values of
$h$) by introducing an artificial cut-off bound in the potential when necessary. 

In view of these considerations, we next limit our exploration to symmetric-conjugate splitting methods of the form (\ref{split1}) with $0 < a_j < 1$ and
$b_j \in \mathbb{C}$ with $\Re(b_j) > 0$ to try to reduce the size of the error terms appearing in the asymptotic expansion of the modified
Hamiltonian $H_h$ associated with the integrator. 

For simplicity, we denote the {symmetric-conjugate} splitting schemes $S_h$  by their sequence of coefficients as 
\begin{equation} \label{aba}
  (a_0, b_0, a_1, b_1, \ldots, a_r, b_r, a_r, \ldots, \overline{b}_1, a_1, \overline{b}_0,  a_0).
\end{equation}
As a matter of fact, since $A$ and $B$ are sought to verify (\ref{rkn}), sequences starting with $B$ may lead to schemes with a different efficiency, 
so that we also analyze methods of the form
\begin{equation} \label{bab}
  (b_0, a_0, b_1, a_1, \ldots, b_r, a_r, \overline{b}_r, \ldots, a_1, \overline{b}_1, a_0, \overline{b}_0).
\end{equation}
Schemes (\ref{aba}) and (\ref{bab}) include integrators where the central exponential corresponds to $A$ (when $b_r=0$) and $B$ (when $a_r=0$),
respectively. The method has $s$ stages if the number of exponentials of $A$ is precisely $s$ for the scheme (\ref{bab}) or $s+1$ for the scheme (\ref{aba}).

%\footnote{Notice that in (\ref{bab}) the last exponential in
%the current time step may be used in the next step.}.}

The construction process of methods within this class is detailed elsewhere (e.g. \cite{blanes08sac,blanes22rkn} and references therein), so that
it is only summarized here.
 First, we get the order conditions a symmetric-conjugate scheme has
to satisfy to achieve a given order $p=4, 5$ and 6. These are polynomial equations depending on the coefficients $a_j$, $b_j$, 
and can be obtained by identifying a basis in the Lie algebra generated by $\{A, B\}$ and using repeatedly 
the BCH formula to express the splitting method as $S_h = \exp(h H_h)$, with $H_h$ in terms of $A$, $B$ and their nested commutators.
The order conditions up to order $p$ are obtained by requiring that $H_h = H + \mathcal{O}(h)^{p+1}$, and the number is 7, 11 and 16 for orders 4, 5 and 6, respectively.   

Second, we take compositions (\ref{aba}) and (\ref{bab}) involving the minimum number of stages required to solve the order conditions and get eventually
all possible solutions with the appropriate symmetry. Sometimes, one has to add parameters, because there are no such solutions. 
In particular, there are no 4th-order schemes with 4 stages with both $a_j > 0$ and $\Re(b_j) > 0$.

Even when there are appropriate solutions, it may be convenient to explore compositions with additional stages to have free parameters for optimization. 
This strategy usually pays off
when purely real coefficients are involved, and so it is worth to be explored also in this context. 
Of course, some optimization criterion related with the error terms and the computational effort has to be adopted. 
In our study we look at the error terms in the expansion of $H_h$ at successive orders and the size of the $b_j$ coefficients.
Specifically, we compute for each method of order, say, $p$, the quantities 
\begin{equation} \label{delta}
   \Delta_b := \sum_j |b_j| \qquad \mbox{ and } \qquad E_f^{(r+1)} := s \, \big( \mathcal{E}_{r+1} \big)^{1/r}, \qquad r = p, p+1, \ldots
\end{equation}   
Here $s$ is the number of stages and $\mathcal{E}_{r+1}$ is the Euclidean norm of the vector of error coefficients in $H_h$ at higher orders than the
method itself. In particular, for a method of order $6$, $E_f^{(7)}$ gives an estimate of the efficiency of the scheme by considering only the
error at order 7. By computing  $E_f^{(8)}$ and $E_f^{(9)}$ for this method we get an idea of how the higher order error terms behave. It will be of interest,
of course, to reduce these quantities as much as possible to get efficient schemes.

Solving the polynomial equations required to construct splitting methods with additional stages is not a trivial task, especially for orders 5 and 6. In these
cases we have used the Python function \texttt{fsolve} of the \emph{SciPy} library, with a large number of initial points in the space of parameters 
to start the procedure. From the total number of valid solutions thus obtained, we have selected those leading
to reasonably small values of all quantities (\ref{delta}) and checked them on numerical examples.

The corresponding values for the most efficient methods we have found by following this approach have been collected in Table \ref{tau:efrknsc}, 
where $\mathcal{NA}_{s}^{*[p]}$ refers to a symmetric-conjugate method of type (\ref{aba}) of order $p$ involving $s$ stages, and $\mathcal{NB}_{s}^{*[p]}$ 
is a similar scheme of type (\ref{bab}). For completeness, we have also included the most efficient integrators of order 4, 6 and 8 with
real coefficients for systems satisfying the condition (\ref{rkn}) (same notation without $*$) and also the 
symmetric-conjugate splitting schemes presented in \cite{goth20hoa,goth22hoa}
(denoted by $\mathcal{GB}_{s}^{*[p]}$). They do not take into account the property (\ref{rkn}) for their formulation. 

In Table \ref{tau:efrknsc} we also write the value of $\Delta_a := \sum_j |a_j|$ and $\Delta_b := \sum_j |b_j|$ for each method. 
Of course, by construction, $\Delta_a = 1$ for all
symmetric-conjugate integrators. The coefficients of the most  efficient schemes we have found (in boldface) are collected in Table \ref{coefRKN}.

\begin{table}[!h]
{\small
  \renewcommand\arraystretch{1.4}
  \begin{center}
    \begin{tabular}{c|ccccccc|}
      &$\Delta_a$&$\Delta_b$&$E_f^{(5)}$&$E_f^{(6)}$&$E_f^{(7)}$&$E_f^{(8)}$&$E_f^{(9)}$\\
      \cline{1-8}   
       $\mathcal{NA}_{6}^{*[4]}$&1.000&1.267& 0.400& 0.821& 0.704& 1.082& 1.012\\
      $\boldsymbol{\mathcal{NB}_{5}^{*[4]}}$&\textbf{1.000}&\textbf{1.141}& \textbf{0.352}& \textbf{0.698}& \textbf{0.559}& \textbf{0.913}& \textbf{0.789}\\
      $\boldsymbol{\mathcal{NB}_{6}^{*[4]}}$&\textbf{1.000}&\textbf{1.416}& \textbf{0.322}& \textbf{0.766}& \textbf{0.666}& \textbf{1.025}& \textbf{0.866}\\
      $\mathcal{NA}_{7}^{*[5]}$&1.000&1.662& --& 0.695& 0.817& 1.013& 1.132\\
      $\mathcal{NA}_{8}^{*[5]}$&1.000&1.393&-- &0.546& 0.947& 0.953& 1.339\\
      $\mathcal{NA}_{9}^{*[5]}$&1.000&1.456&-- &0.498& 0.970& 1.157& 1.357\\
      $\mathcal{NB}_{7}^{*[5]}$&1.000&3.196&-- &0.833& 0.970& 1.143& 1.300\\
      $\boldsymbol{\mathcal{NB}_{8}^{*[5]}}$&\textbf{1.000}&\textbf{1.482}&-- &\textbf{0.478}& \textbf{0.670}& \textbf{1.046}& \textbf{1.031}\\
      $\boldsymbol{\mathcal{NB}_{9}^{*[5]}}$&\textbf{1.000}&\textbf{1.618}&-- &\textbf{0.403}& \textbf{0.966}& \textbf{1.331}& \textbf{1.499}\\       
      $\mathcal{NA}_{10}^{*[6]}$&1.000&1.528&-- &--& 0.906& 1.204& 1.298\\
       $\boldsymbol{\mathcal{NA}_{11}^{*[6]}}$&\textbf{1.000}&\textbf{2.092}&--&--&\textbf{0.656}&\textbf{1.418}&\textbf{1.643}\\      
      $\mathcal{NB}_{10}^{*[6]}$&1.000&1.516&-- &-- &1.000&1.212&1.557\\
      $\boldsymbol{\mathcal{NB}_{11}^{*[6]}}$&\textbf{1.000}&\textbf{1.595}&--&--&\textbf{0.646}&\textbf{1.387}&\textbf{1.394}\\      
      \cline{1-8}
      $\mathcal{GB}_{5}^{*[4]}$&1.000&1.133& 0.477 &0.662& 0.662& 0.885& 0.807\\
      $\mathcal{GB}_{9}^{*[5]}$&1.000&1.463& --&0.603& 0.786& 1.036& 1.278\\
      $\mathcal{GB}_{15}^{*[6]}$&1.000&1.692&-- &--&1.515&1.434&2.169 \\     
      \cline{1-8} 
     % \hline\hline
      $\mathcal{NB}_{6}^{[4]}$&2.401&1.156& 0.291& --& 0.809& --& 1.307\\
      $\mathcal{NB}_{11}^{[6]}$&2.494&1.206&--& --& 0.784& --& 1.664\\
      $\mathcal{NA}_{14}^{[6]}$&1.659&2.012& --& --& 0.627& --& 2.238\\
      \cline{1-8}
    \end{tabular}
    \caption{1-norm and effective errors for several splitting methods of order 4, 5 and 6 designed for problems satisfying the condition (\ref{rkn}).}
    \label{tau:efrknsc}
  \end{center}
  }
\end{table}

\begin{table}[!h]
{\small
  \renewcommand\arraystretch{1.4}
  \begin{center}
    \begin{tabular}{lll|}
      &\multicolumn{1}{c}{$a_i$}&\multicolumn{1}{c}{$b_i$} \\ 
      \cline{2-3} % obtained by Fernando
      $\mathcal{NB}_{5}^{*[4]}$&$a_0=0.17354158169943656$ &$b_0=0.06421454120274125 + 0.0245540186592381\, i$ \\
      &$a_1=                       0.19379086394173623$ &$b_1=0.20166370500451958 - 0.0982277975564409 \, i$ \\
      &$a_2=1-2\sum_{i=0}^1a_i$ &$b_2=\frac{1}{2}-\sum_{i=0}^1\Re(b_i) +             0.1491719824749133 \, i$ \\
      \cline{2-3}
%      &\\
%      &\multicolumn{1}{c}{$a_i$}&\multicolumn{1}{c}{$\Re(b_i)$}&\multicolumn{1}{c}{$\Im(b_i)$} \\ 
      \cline{2-3}
      $\mathcal{NB}_{6}^{*[4]}$&$a_0=\frac{1}{5}$ & $b_0=\frac{7}{100} + 0.019444288930263294 \, i$ \\
      &$a_1=0.054855282174763084$ &$b_1=0.16 -0.20579973912385285 \, i$ \\
      &$a_2=\frac{1}{2}-\sum_{i=0}^1a_i$ &$b_2=0.16251793145097668 + 0.21219211957584155 \, i$ \\
      & &$b_3=1-2\sum_{i=0}^1\Re(b_i)$ \\
      \cline{2-3}  
%      &\\
%      &\multicolumn{1}{c}{$a_i$}&\multicolumn{1}{c}{$\Re(b_i)$}&\multicolumn{1}{c}{$\Im(b_i)$} \\ 
      \cline{2-3}
      $\mathcal{NB}_{8}^{*[5]}$&$a_0=0.13556579817637690$ & $b_0=0.048 -0.0045117121645322032 \, i$ \\
      &$a_1=0.12110548685533656$ &$b_1=0.159 + 0.039915395925895825 \, i$ \\
      &$a_2=0.040926280383255811$ &$b_2=0.08808186616153123 -0.19475521098317861 \, i$ \\
      &$a_3=\frac{1}{2}-\sum_{i=0}^2\Re(a_i)$ &$b_3=0.08139005735125036 + 0.17341123352295854 \, i$\\
      &                                     &$b_4=1-2\sum_{i=0}^3b_i$\\ 
      \cline{2-3}
%      &\\
%     &\multicolumn{1}{c}{$a_i$}&\multicolumn{1}{c}{$\Re(b_i)$}&\multicolumn{1}{c}{$\Im(b_i)$} \\ 
      \cline{2-3}
      $\mathcal{NB}_{9}^{*[5]}$&$a_0=0.066$ &$b_0=0.03 -0.026088775868557137 \, i$ \\
      &$a_1=0.066$ &$b_1=0.065 + 0.0871906864166141 \, i$ \\
      &$a_2=0.15406042184345631$ &$b_2=0.087791471011534450 -0.07869869176637824 \, i$ \\
      &$a_3=0.20434260458660722$ &$b_3=0.21903826707051549 + 0.005649631789653575 \, i$\\
      &$a_4=1-2\sum_{i=0}^3a_i$ &$b_4=\frac{1}{2}-\sum_{i=0}^3\Re(b_i) + 0.3080209334852549 \, i$\\ 
      \cline{2-3}      
%       &\\
%     &\multicolumn{1}{c}{$a_i$}&\multicolumn{1}{c}{$\Re(b_i)$}&\multicolumn{1}{c}{$\Im(b_i)$} \\ 
      \cline{2-3}
      $\mathcal{NA}_{11}^{*[6]}$&$a_0=0.062770091$ &$b_0=0.10891717046144 -0.16165289456182 \, i$ \\
      &$a_1=0.011912916558090$ &$b_1=0.05673774365156 + 0.19084324113721 \, i$ \\
      &$a_2=0.20435669618321$ &$b_2 =0.00000000664446 -0.2132590752834 \, i$ \\
      &$a_3=0.019233264988143$ &$b_3=0.2404799796837 + 0.10112304441789 \, i$ \\
      &$a_4=0.06593857714457$ &$b_4=0.04313692053520 + 0.11954730647763 \, i$ \\
      &$a_5=\frac{1}{2}-\sum_{i=0}^4a_i$ &$b_5=1-2\sum_{i=0}^4\Re(b_i)$\\ 
      \cline{2-3}        
%      &\\
%      &\multicolumn{1}{c}{$a_i$}&\multicolumn{1}{c}{$\Re(b_i)$}&\multicolumn{1}{c}{$\Im(b_i)$} \\ 
      \cline{2-3}
      $\mathcal{NB}_{11}^{*[6]}$&$a_0=\frac{213}{2500}$ &$b_0=\frac{7}{250} -0.009532915454170 \, i$ \\
      &$a_1=0.047358568390005$ &$b_1=0.08562523731685 + 0.0718344013568 \, i$ \\
      &$a_2=0.1553620075936$ &$b_2=0.09331583397900 -0.09161071812994 \, i$ \\
      &$a_3=0.10012117440925$ &$b_3=0.11799012127542 + 0.0702739287203 \, i$ \\
      &$a_4=0.10547836949919$ &$b_4=0.16176918420712 -0.04327349898459 \, i$ \\
      &$a_5=1-2\sum_{i=0}^4a_i$ &$\Re(b_5)=\frac{1}{2}-\sum_{i=0}^4 \Re(b_i) -0.2203293328195 \, i$\\
      \cline{2-3} 
    \end{tabular}
    \caption{Coefficients of the most efficient symmetric-conjugate RKN splitting methods of order 4, 5 and 6.}
    \label{coefRKN}
  \end{center}
  }
\end{table}

In the Appendix we provide analogous information for general schemes of orders 3, 4, 5 and 6, i.e., of splitting methods
for general problems of the form
$H = A + B$, with $a_j > 0$ and
$b_j \in \mathbb{C}$ with $\Re(b_j) > 0$. They typically involve more stages, but can be applied in more general contexts. 

{One should take into account, however, that all these symmetric-conjugate methods have been obtained by considering the ordinary differential
equation (\ref{problem1}) in finite dimension, whereas the time dependent Schr\"odinger equation is a prototypical example of an evolutionary
PDE involving unbounded
operators (the Laplacian and possibly the potential). In consequence, one might arguably question the viability of using the above schemes in this setting. That
this is indeed possible comes as a consequence of some previous results obtained in the context of PDEs defined in analytic semigroups.

Specifically, equation (\ref{Schr0}) can be written in the generic form 
\begin{equation} \label{gen.1}
  u' = \hat{L} u = (\hat{A}+\hat{B}) u, \qquad u(0) = u_0,
\end{equation}
with $\hat{A} = \frac{i}{2} \Delta$ and 
$\hat{B} = -i \hat{V}$. It has been
shown in \cite{hansen09esf} (see also \cite{jahnke00ebf,thalhammer12cao}) that, under the two assumptions stated below, a splitting method of the form
\begin{equation} \label{sp.ge}
  S_h = \e^{h a_0 \hat{A}} \, \e^{ h b_0 \hat{B}} \, \cdots  \, \e^{ h b_{2n-1} \hat{B}} \, \e^{ h a_{2n} \hat{A}} 
\end{equation}
is of order $p$ for problem (\ref{gen.1}) if and only if it is of classical order $p$ in the finite dimensional case. The assumptions are as follows:
\begin{enumerate}
\item {\em Semi-group property}: $\hat{A}$, $\hat{B}$ and $\hat{L}$ generate $C^0$-semigroups on a Banach space 
$X$ with norm $\| \cdot \|$ and, in addition, they
satisfy the bounds 
\begin{eqnarray*}
\|\e^{t \hat{A}} \| \leq \e^{\omega t},  \qquad \|\e^{t \hat{B}} \| \le \e^{\omega t} 
\end{eqnarray*}
for some positive constant $\omega$ and all $t \ge 0$.
\item {\em Smoothness property}: For any pair of multi-indices $(i_1,\ldots,i_m)$ and $(j_1,\ldots,j_m)$ with $i_1+ \cdots + i_m + j_1 + \cdots + j_m = p+1$, and for all $t \in [0,T]$,
\begin{eqnarray*}
\|\hat{A}^{i_1} \hat{B}^{j_1} \ldots \hat{A}^{i_m} \hat{B}^{j_m} \, \e^{t \hat{L}}u_0\| \leq C
\end{eqnarray*}
for a positive constant $C$.
\end{enumerate}
These conditions restrict the coefficients $a_j$, $b_j$ in (\ref{sp.ge}) to be positive, however, and thus the method to be of second order at most. 
Nevertheless, it has been shown in \cite{hansen09hos,castella09smw} that,
if in addition $\hat{L}$, $\hat{A}$ and $\hat{B}$ generate analytic semigroups on $X$ defined in the sector $\Sigma_{\phi} = \{ z \in \mathbb{C} :
|\arg z| < \phi \}$, for a given angle $\phi \in (0, \pi/2]$ and the operators $\hat{A}$ and $\hat{B}$ verify
\[
  \|\e^{z \hat{A}} \| \leq \e^{\omega |z|},  \qquad \|\e^{z \hat{B}} \| \le \e^{\omega |z|} 
\]
for some $\omega \ge 0$ and all $z \in \Sigma_{\phi}$, then a splitting method of the form (\ref{sp.ge}) of classical order $p$ with all its coefficients
$a_j$, $b_j$ in the sector $\Sigma_{\phi} \subset \mathbb{C}$, then
\[
  \| (S_h^n - \e^{n h \hat{L}}) u_0 \| \le C h^p, \qquad 0 \le n h \le T
\]
where $C$ is a constant independent of $n$ and $h$.

%These results guarantee that our symmetric-conjugate methods can be safely applied to the Schr\"odinger equation.  
}
 
\section{Numerical illustration: Modified P\"oschl--Teller potential}
\label{section5}

%\paragraph{Modified P\"oschl--Teller potential.}
The so-called modified P\"oschl--Teller potential takes the form
\begin{equation} \label{pt1}
  V(x) = -\frac{\alpha^2}{2} \frac{\lambda (\lambda-1)}{\cosh^2 \alpha x},
\end{equation}
with $\lambda > 1$, and admits an analytic treatment to compute explicitly the eigenvalues for negative energies \cite{flugge71pqm}. For the simulations
we take $\alpha = 1$, $\lambda(\lambda-1) = 10$ and the initial condition $\psi_0(x) = \sigma \, \e^{-x^2/2}$, with $\sigma$ a normalizing constant.
We discretize the interval $x \in [-8,8]$ with $N=256$ equispaced points and apply Fourier spectral methods. With this value of $N$ it turns out that
$\| ([B,[B,[A,B]]]) u_0\|$ is sufficiently close to zero to be negligible, so that we can safely apply the schemes of Table \ref{coefRKN}. If $N$ is not sufficiently
large, then the corresponding matrices $A$ and $B$ do not satisfy (\ref{rkn}), and as a consequence, the schemes are only of order three. This can be
indeed observed in practice.

We first check how the errors in the norm $\mathcal{M}(u)$ and in the energy $\mathcal{H}(u)$ evolve with time according with each type of integrator. To
this end we integrate numerically until the final time $t_f = 10^4$ with three 6th-order compositions involving complex coefficients: (i) the new
symmetric-conjugate scheme $\mathcal{NB}_{11}^{*[6]}$ collected in Table \ref{coefRKN} $(h = 100/909 \approx 0.11)$, 
(ii) the palindromic scheme denoted by $\mathcal{B}_{16}^{[6]}$
with all $a_j$ taking the same value $a_j = 1/16$, $j=1,\ldots, 8$ and complex $b_j$ with positive real part\footnote{The coefficients can be found at the website
\url{http://www.gicas.uji.es/Research/splitting-complex.html}.} $(h=0.16)$, and (iii) the method obtained by composing $\mathcal{B}_{16}^{[6]}$ with its complex
conjugate $(\mathcal{B}_{16}^{[6]})^*$, resulting in a symmetric-conjugate integrator $(h=0.32)$.  
The step size is chosen in such a way that all the 
methods require the same number of FFTs. The results  are depicted in Figure \ref{figure-error1}. We see that, according
with the previous analysis, {the error in both unitarity and energy furnished by the new scheme $\mathcal{NB}_{11}^{*[6]}$ does not grow with time}, 
in contrast with
palindromic compositions involving complex coefficients. Notice also that the composition of the palindromic scheme $\mathcal{B}_{16}^{[6]}$ with its
complex conjugate leads to a new (symmetric-conjugate) integrator with good preservation properties. On the other hand, composing a symmetric-conjugate
method with its complex conjugate results in a palindromic scheme showing a drift in the error of both the norm and the energy \cite{blanes22asm}.

\begin{figure}[!ht] 
\centering
  \includegraphics[width=.49\textwidth]{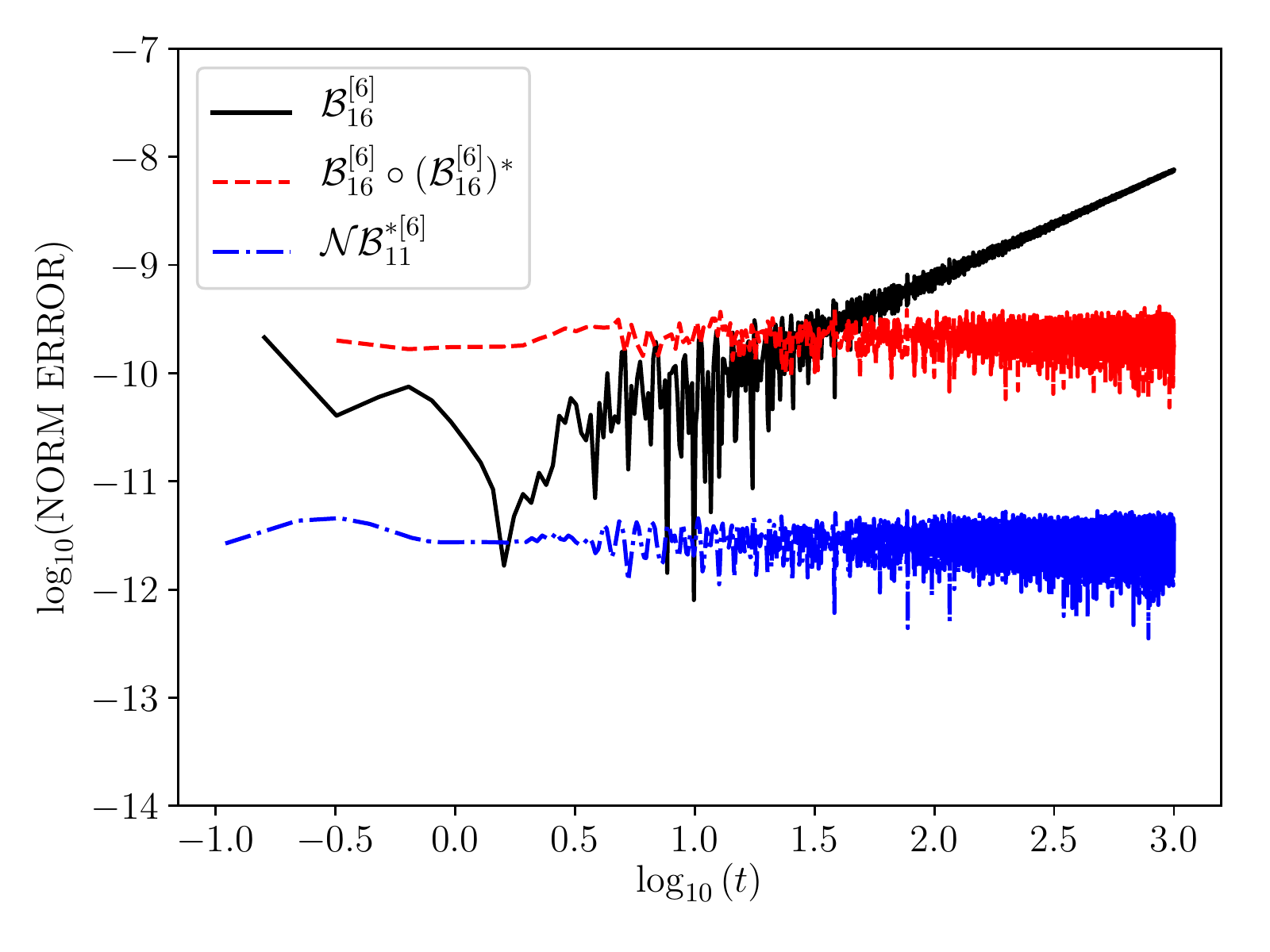}
  \includegraphics[width=.49\textwidth]{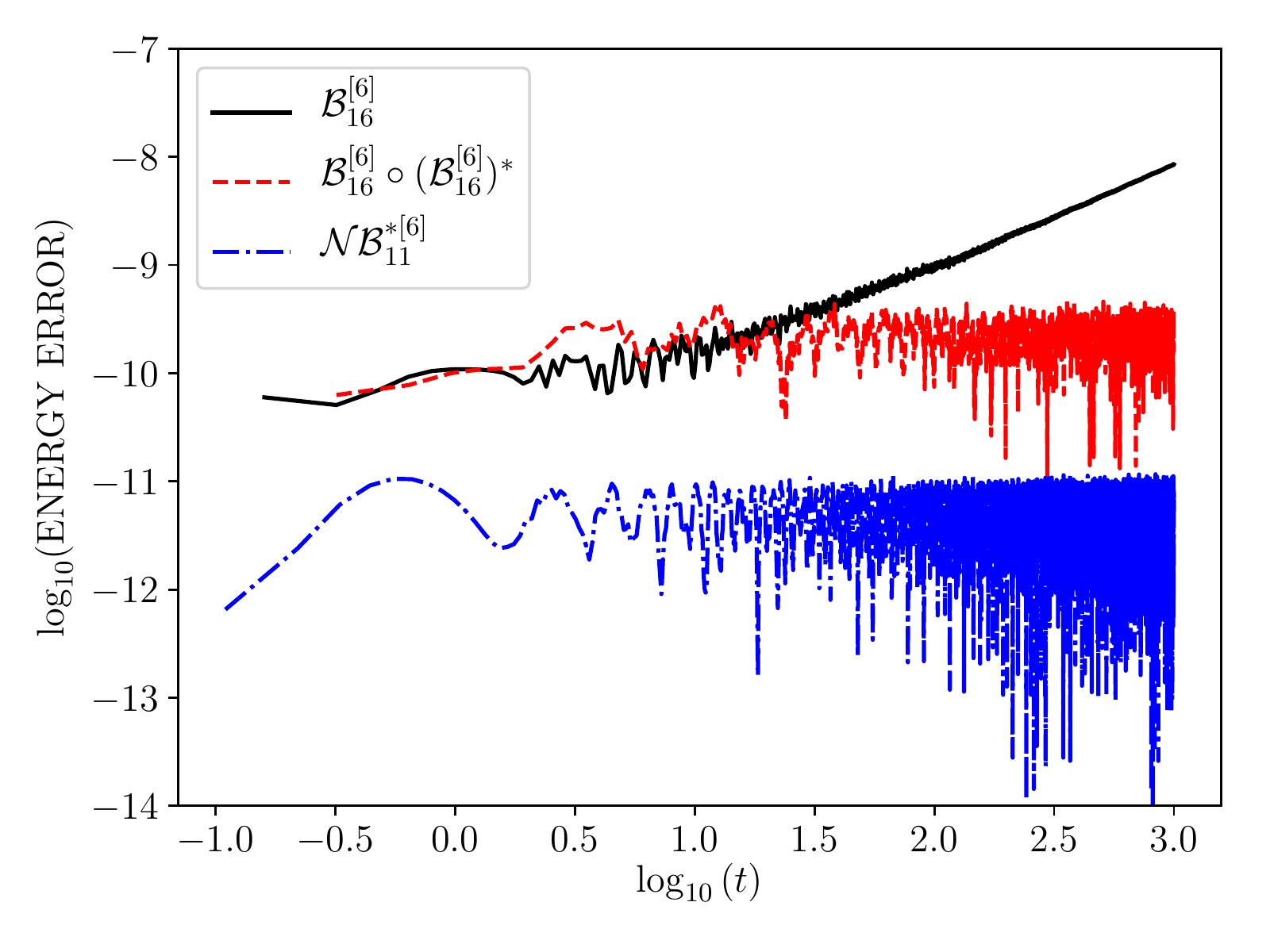}
 \caption{\label{figure-error1} \small Error in norm $\mathcal{M}(u)$ (left) and in energy $\mathcal{H}(u)$ (right) as a function of time for complex-conjugate
 and palindromic methods involving complex coefficients.
}
\end{figure}

In our second experiment, we test the efficiency of the different schemes. To this end we integrate until the final time $t_f = 100$, compute 
the expectation value of the energy,$\mathcal{H}(u_{\mathrm{app}}(t))$, and measure the error as the maximum of the 
difference with respect to the exact value along the integration:
\begin{equation} \label{eq.5.1b}
    \max_{0 \le t \le t_f} \quad |\mathcal{H}(u_{\mathrm{app}}(t))  - \mathcal{H}(u_0)|.
\end{equation}   
The corresponding results are displayed as a function of the computational
cost measured by the number of FFTs necessary to carry out the calculations (in log-log plots) in  Figure \ref{figure4.1}. Notice how the new 
symmetric-conjugate schemes offer a better efficiency than standard splitting methods for this problem. The improvement is particularly significant in the
6th-order case.

\begin{figure}[!ht] 
\centering
  \includegraphics[width=.49\textwidth]{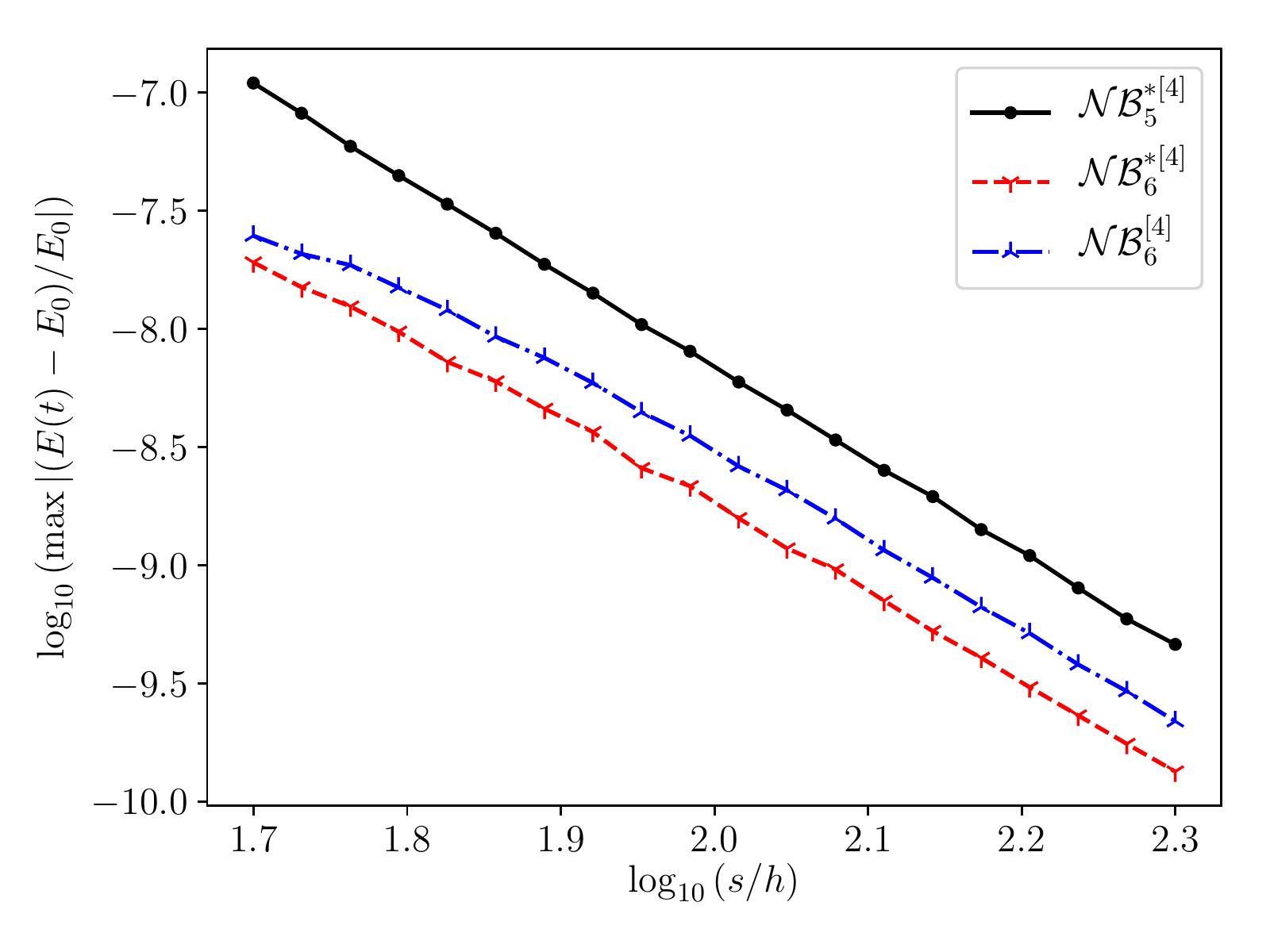}
  \includegraphics[width=.49\textwidth]{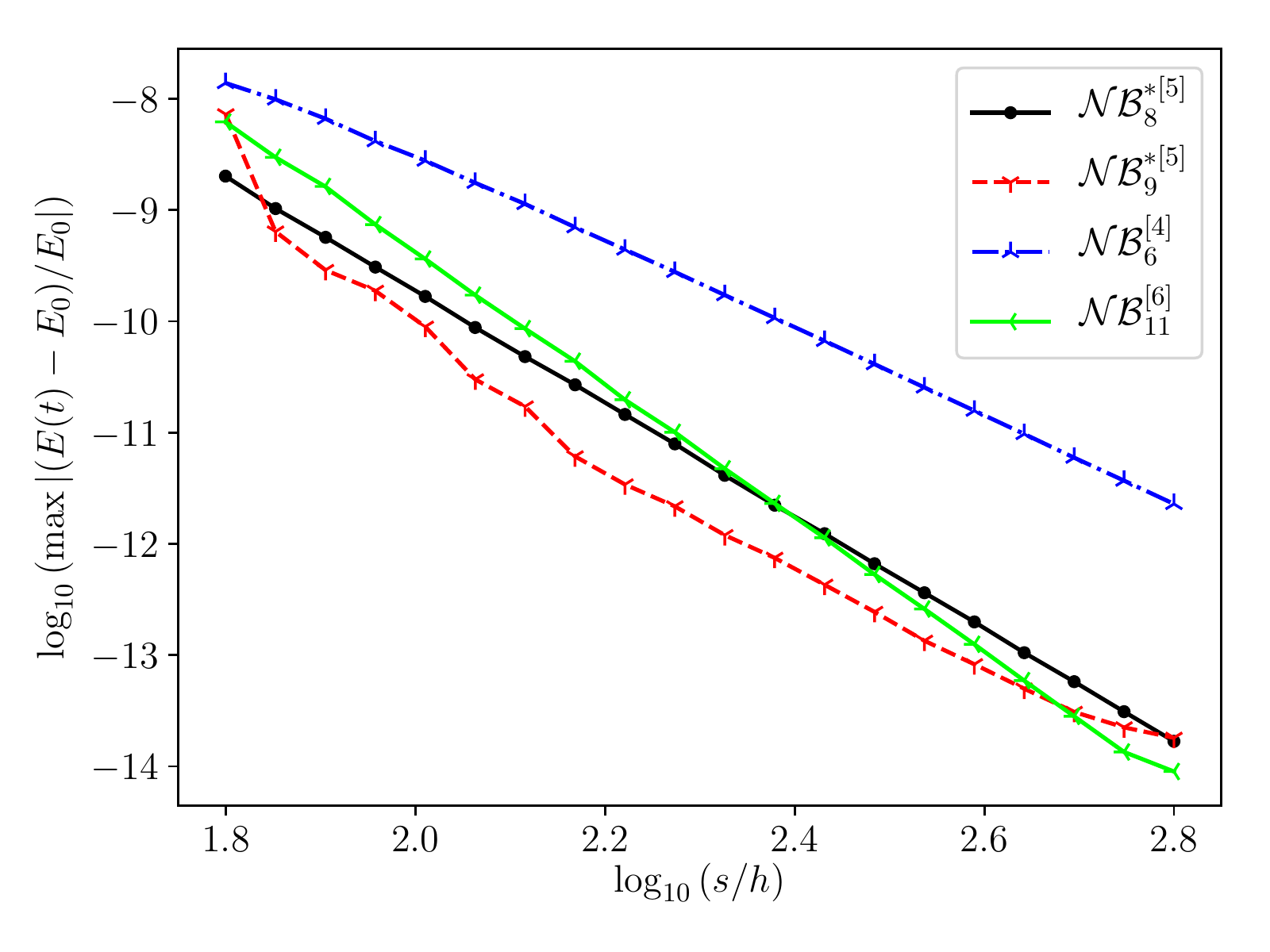}
  \includegraphics[width=.49\textwidth]{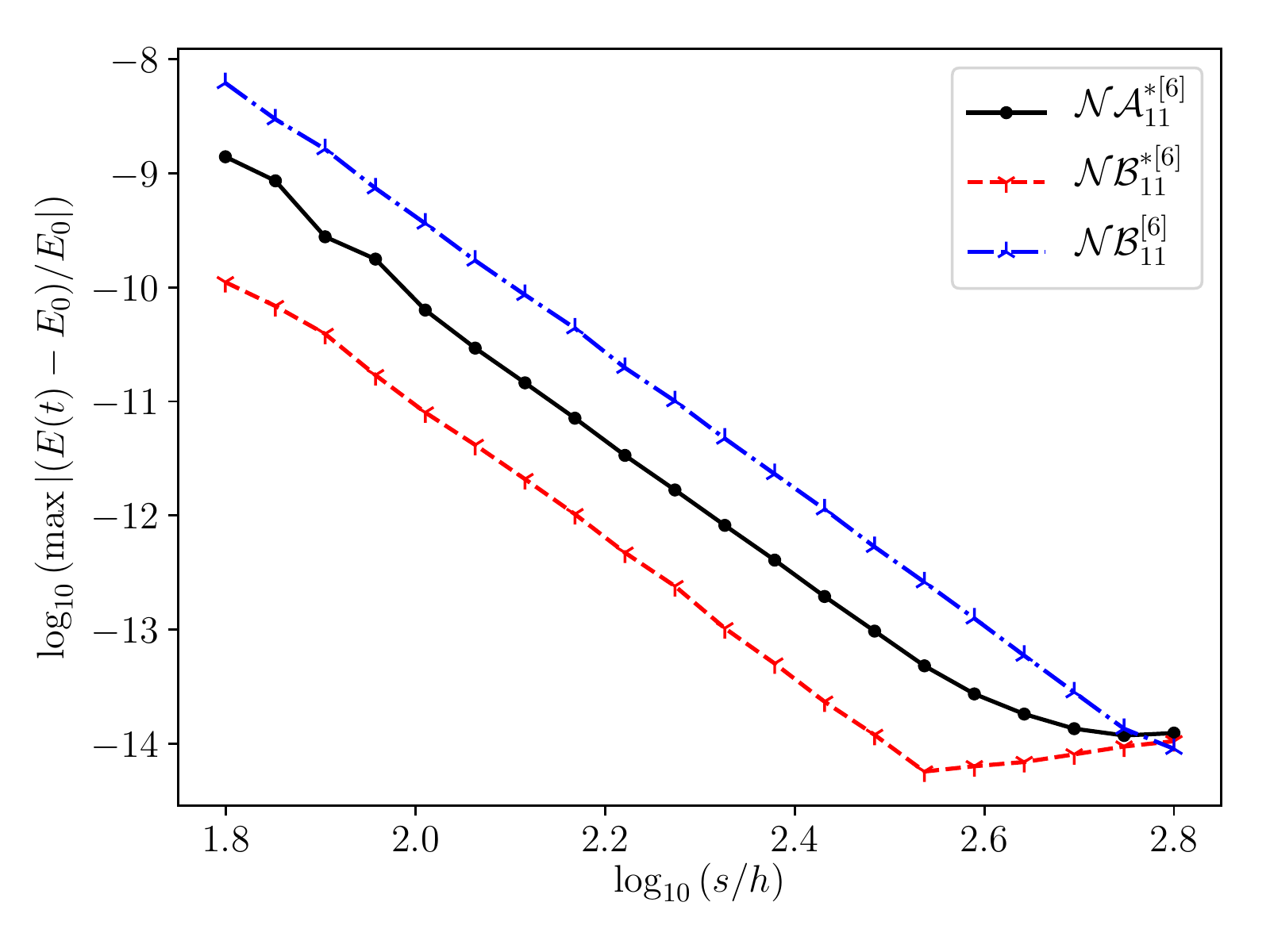}
\caption{\label{figure4.1} \small Maximum error in the expectation value of the energy along the integration for several 4th-, 5th- and 6th-order symmetric-conjugate
splitting methods for the modified P\"oschl--Teller potential.
}
\end{figure}

\subsection*{Acknowledgements}
%%%%%%%%%%%%%%%%%%%%%%%%%%%%%%%%%%%%%%%%%%%%%%%%%%%%%%%%%%%%%%%%%%%%%%%%%%%%%%%%%%%%%%%%%%%%%%%%%%%%%%%%%%%%%%%%%%%%%
The work of JB is supported by ANR-22-CE40-0016 ``KEN" of the Agence Nationale de la Recherche (France) and by the region Pays de la 
Loire (France) through the project ``MasCan". SB, FC and AE-T acknowledge financial support by 
Ministerio de Ciencia e Innovaci\'on (Spain) through project PID2019-104927GB-C21, MCIN/AEI/10.13039/501100011033, ERDF (``A way of ma\-king Europe"). 
The authors would also like to thank Prof. C. Lubich for his very useful remarks. 

\subsection*{Compliance with Ethical Standards}

All authors declare that they have no conflicts of interest.

\appendix

\section{Appendix}
  
 We collect in this Appendix the most efficient symmetric-conjugate splitting methods with  $a_j > 0$ and
$b_j \in \mathbb{C}$ with $\Re(b_j) > 0$ we have found for a general problem of the form $H = A + B$. The coefficients of the schemes in boldface
in Table \ref{general} are listed in Table \ref{coef-general}. Methods of type (\ref{aba}) of order $p$ involving $s$ stages are denoted as 
$\mathcal{A}_{s}^{*[p]}$, whereas $\mathcal{B}_{s}^{*[p]}$  refers to a similar scheme of type (\ref{bab}). 
As in Table \ref{tau:efrknsc}, we also collect for reference the methods proposed in \cite{goth20hoa} (denoted by $\mathcal{GB}_{s}^{*[p]}$) and two
efficient palindromic compositions of time-symmetric schemes of order 2 with real coefficients, $\mathcal{S}_s^{[p]}$. At order 5, the most efficient
scheme turns out to be $\mathcal{GB}_{9}^{*[5]}$.

We also include a numerical illustration on the modified P\"oschl--Teller potential with the same data as before. 
Notice in particular the improvement with respect to the 6th-order
scheme $\mathcal{S}_{10}^{[6]}$.

\begin{table}[!ht]
{\small
  \renewcommand\arraystretch{1.4}
  \begin{center}
    \begin{tabular}{l|llllllll|}
      &$\Delta_a$&$\Delta_b$&$E_f^{(4)}$&$E_f^{(5)}$&$E_f^{(6)}$&$E_f^{(7)}$&$E_f^{(8)}$&$E_f^{(9)}$\\
      \cline{2-9}
      $\boldsymbol{\mathcal{B}_{3}^{*[3]}}$&\textbf{1.000}&\textbf{1.766}& \textbf{0.522}& \textbf{0.509} &\textbf{0.682}& \textbf{0.664}& \textbf{0.812}& \textbf{0.875}\\
      $\mathcal{A}_{6}^{*[4]}$&1.000&1.125& --& 0.410 &0.827& 0.608& 1.090& 0.988\\
      $\boldsymbol{\mathcal{B}_{5}^{*[4]}}$&\textbf{1.000}&\textbf{1.146}& --& \textbf{0.399} &\textbf{0.764}& \textbf{0.569}& \textbf{0.972}& \textbf{0.772}\\
      $\mathcal{B}_{6}^{*[4]}$&1.000&1.136& --& 0.445 &0.911& 0.626& 1.158& 0.881\\
      $\mathcal{A}_{9}^{*[5]}$&1.000&1.704& --& --&1.141& 1.173& 1.521& 1.744\\
      $\mathcal{B}_{9}^{*[5]}$&1.000&1.480&-- &-- &0.885& 0.826& 1.198& 1.493\\      
      $\mathcal{A}_{15}^{*[6]}$&1.000&1.355&-- &-- &--&1.544&1.335&2.348 \\
      $\boldsymbol{\mathcal{B}_{15}^{*[6]}}$&\textbf{1.000}&\textbf{1.327}&-- &-- &-- &\textbf{1.150}&\textbf{1.274}&\textbf{2.116}\\
      \cline{1-9}
      $\mathcal{GB}_{3}^{*[3]}$&1.000&1.155& 0.586& 0.445& 0.722& 0.642& 0.777& 0.772 \\
      $\mathcal{GB}_{5}^{*[4]}$&1.000&1.133& --& 0.480 &0.698& 0.676& 0.918& 0.830\\
      $\boldsymbol{\mathcal{GB}_{9}^{*[5]}}$&\textbf{1.000}&\textbf{1.463}& --& --&\textbf{0.681}& \textbf{0.819}& \textbf{1.126}& \textbf{1.439}\\
      $\mathcal{GB}_{15}^{*[6]}$&1.000&1.692&-- &-- &--&1.583&1.445&2.361 \\
      \cline{1-9}
      $\mathcal{S}_{6}^{[4]}(aba)$&1.168&1.575& --&0.559& --& 0.792& --& 1.239\\
      $\mathcal{S}_{10}^{[6]}(aba)$&3.203&1.595& --& --& --&1.144& --& 1.606\\
      \cline{1-9}
    \end{tabular}
    \caption{1-norm and effective errors for symmetric-conjugate splitting methods for $H = A + B$.}
    \label{general}
  \end{center}
  }
\end{table}

\begin{table}[!h]
{\small
  \renewcommand\arraystretch{1.4}
  \begin{center}
    \begin{tabular}{lll|}
      &\multicolumn{1}{c}{$a_i$}&\multicolumn{1}{c}{$b_i$} \\ 
      \cline{2-3}
      $\mathcal{B}_{3}^{*[3]}$&$a_0=0.4706$ &$b_0=0.1655101882118 + 0.03704896872215 \, i$ \\
      &$a_1=1-2a_0$ &$\Re(b_1)=\frac{1}{2}-\Re(b_0) - 0.6300845020773 \, i$\\ 
      \cline{2-3}
%      &\\
%      &\multicolumn{1}{c}{$a_i$}&\multicolumn{1}{c}{$\Re(b_i)$}&\multicolumn{1}{c}{$\Im(b_i)$} \\ 
      \cline{2-3}
      $\mathcal{B}_{5}^{*[4]}$&$a_0=\frac{37}{250}$ &$b_0=0.05338438633498185 -0.03218942894140047 \, i$ \\
                                                       &$a_1=0.22446218092466344$ &$b_1=0.19561815336463223 + 0.0992879758243923 \, i$ \\
      &$a_2=1-2\sum_{i=0}^1a_i$ &$b_2=\frac{1}{2}-\sum_{i=0}^1\Re(b_i) -0.14783578044680548 \, i$ \\
      \cline{2-3}
%      &\\
%      &\multicolumn{1}{c}{$a_i$}&\multicolumn{1}{c}{$\Re(b_i)$}&\multicolumn{1}{c}{$\Im(b_i)$} \\ 
      \cline{2-3}    
      $\mathcal{B}_{15}^{*[6]}$&$a_0=0.08092666015955027$ &$b_0=\frac{3}{100} -0.0028985018717006387 \, i$ \\
      &$a_1=0.06736427978832901$ &$b_1=0.08826477458499815 + 0.019065371639195743 \, i$ \\
      &$a_2=0.057276240999706116$ &$b_2=0.07026507350715319 -0.05226928459003309 \, i$ \\
      &$a_3=0.06428730473896961$ &$b_3=0.051044248093469226 + 0.07580262639617709 \, i$ \\
      &$a_4=0.05528732144478408$ &$b_4=0.040506044227148555 -0.07981221177569087 \, i$ \\
      &$a_5=0.02566179136566552$ &$b_5=0.03061653536468681 + 0.07254698089135206 \, i$ \\
      &$a_6=0.10559039215618958$ &$b_6=0.10349890449629792 -0.03539199012223482 \, i$ \\     
      &$a_7=1-2\sum_{i=0}^6a_i$ &$b_7=\frac{1}{2}-\sum_{i=0}^6\Re(b_i) + 0.0111821298374971054 \, i$ \\
      \cline{2-3}
    \end{tabular}
    \caption{Coefficients of the most efficient splitting methods collected in Table \ref{general}.}
    \label{coef-general}
  \end{center}
}  
\end{table}

\begin{figure}[!ht] 
\centering
  \includegraphics[width=.6\textwidth]{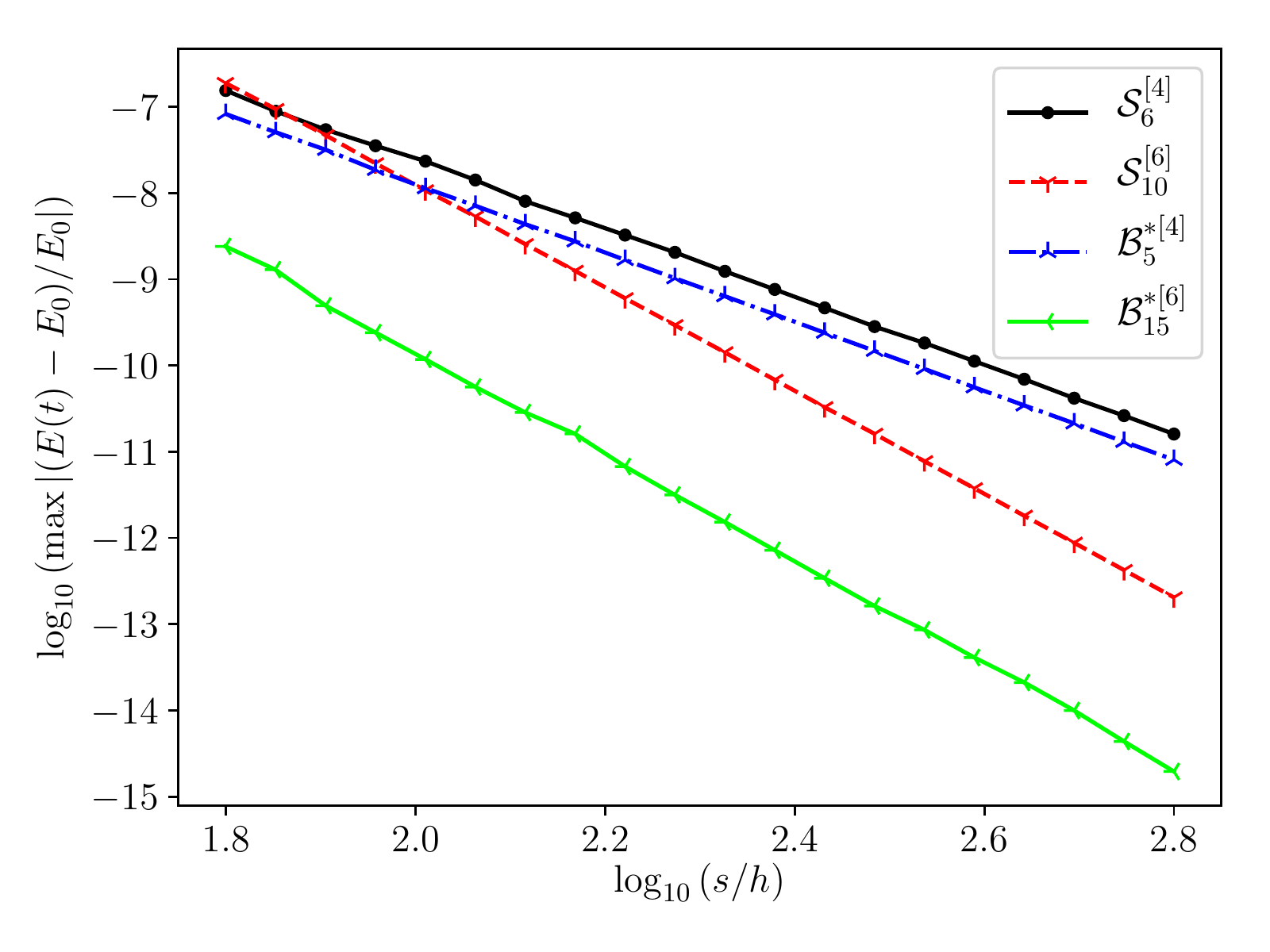}
\caption{\label{figure-app.1} \small Maximum error in the expectation value of the energy along the integration as a function of the computational cost for the new symmetric-conjugate
splitting methods intended for general problems of the form $H = A + B$ (modified P\"oschl--Teller potential).
}
\end{figure}

\bibliographystyle{siam}
%\bibliography{ourbib12-22.bib}

\end{document}